\documentclass[reqno]{amsart} 

\usepackage[utf8]{inputenc}  
\usepackage[cyr]{aeguill}    

\usepackage[english]{babel}

\usepackage[letterpaper,top=2cm,bottom=2cm,left=3cm,right=3cm,marginparwidth=1.75cm]{geometry}

\usepackage{amsmath, amssymb, amsfonts, amsthm, amsxtra, mathrsfs, bbold, mathtools, dsfont}

\usepackage{graphicx}
\graphicspath{{figures/}}  

\usepackage{listings}
\usepackage{float}
\usepackage[ruled,vlined,linesnumbered]{algorithm2e}

\usepackage[colorlinks=true, linkcolor=red, citecolor=red, urlcolor=magenta]{hyperref}
\usepackage{cleveref}

\crefname{theorem}{theorem}{theorems}
\Crefname{theorem}{Theorem}{Theorems}
\crefname{lemma}{lemma}{lemmas}
\Crefname{lemma}{Lemma}{Lemmas}
\crefname{proposition}{proposition}{propositions}
\Crefname{proposition}{Proposition}{Propositions}
\crefname{definition}{definition}{definitions}
\Crefname{definition}{Definition}{Definitions}
\crefname{remark}{remark}{remarks}
\Crefname{remark}{Remark}{Remarks}
\crefname{example}{example}{examples}
\Crefname{example}{Example}{Examples}
\crefname{corollary}{corollary}{corollaries}
\Crefname{corollary}{Corollary}{Corollaries}

\newcounter{Hequation}

\makeatletter\g@addto@macro\equation{\stepcounter{Hequation}}\makeatother
\usepackage{etoolbox}
\newcounter{taggedeq}
\pretocmd{\equation}{\stepcounter{taggedeq}}{}{}

\newtheorem{theorem}{Theorem}[section]

\newtheorem{proposition}{Proposition}[section]

\newtheorem{corollary}{Corollary}[section]
\newtheorem{lemma}{Lemma}[section]

\title[Rates of convergence in alignment models]{Rates of convergence in long time asymptotics of an alignment model with symmetry breaking}

\author[A. Surin]{%
  \normalfont Alexandre Surin\\[6pt] \\
  \small
  Ceremade (CNRS UMR~n\textsuperscript{o}~7534),\\
  PSL University, Université Paris-Dauphine,\\
  Place de Lattre de Tassigny, 75775 Paris 16, France\\
 Email: \texttt{surin@ceremade.dauphine.fr}
}

\date{\today}
\makeatletter
\def\@setauthors{%
  \begingroup
  \def\thanks{\protect\thanks@warning}%
  \trivlist
  \centering\footnotesize \@topsep30\p@\relax
  \advance\@topsep by -\baselineskip
  \item\relax
    \author@andify\authors
    \def\\{\protect\linebreak}%
    {\authors}%
  \endtrivlist
  \endgroup
}
\makeatother

\makeatletter
\def\@settitle{%
  \begin{center}%
    \baselineskip14\p@\relax
    \normalfont\LARGE\@title
  \end{center}%
}
\makeatother

\begin{document}

\begin{abstract}
    We consider a nonlinear Fokker-Planck equation derived from a Cucker-Smale model for flocking with noise. There is a known phase transition depending on the noise between a regime with a unique stationary solution which is isotropic (symmetry) and a regime with a continuum of polarized stationary solutions (symmetry breaking). If the value of the noise is larger than the threshold value, the solution of the evolution equation converges to the unique radial stationary solution. This solution is linearly unstable in the symmetry-breaking range, while polarized stationary solutions attract all solutions with sufficiently low entropy. We prove that the convergence measured in a weighted $L^2$ norm occurs with an exponential rate and that the average speed also converges with exponential rate to a unique limit which determines a single polarized stationary solution.
\end{abstract}
\maketitle

 
  \section{Introduction.} In various models of mathematical biology and social sciences, the dynamics of a population of individuals without leader is often described by mean field interaction forces that depend on the parameters of the agents, for instance, their relative positions and velocities. These interaction rules are the basis for models of collective behavior and the description of emergent structures. There are two  well-known models. The Cucker–Smale model \cite{CuckerSmale2007}  describes agents whose velocities tend to align with the average velocity of their  neighbors, with an interaction strength that decreases with distance. The Vicsek model~\cite{Vicsek1995}, describes particles moving at a constant speed but the direction of the velocity is determined by the average orientation of the neighbors and by a stochastic component, typically a Brownian noise. This noise takes into account individual randomness and estimation errors. The model exhibits a noise-dependent phase transition, marking the emergence of global alignment as the noise intensity decreases. We study a variant of the Cucker–Smale model with Brownian noise and a self-propulsion force. This framework has been investigated in ~\cite{BarbaroCanizoDegondphasetransi,Li2021,Pareschinumerique}  and \cite{Tugaut2014}. We consider the spatially homogeneous case, in which the evolution of the density $f(t,v)$ of individuals is described by 
\begin{equation}
\label{PDE}
\frac{\partial f}{\partial t} = D \, \Delta f + \nabla \cdot \big[ \left(\nabla \psi_{\alpha}(v)+v - \mathbf{u}_{f}   \right) f\big]\quad \text{with} \quad \psi_{\alpha}(v) = \frac{\alpha}{4} \,|v|^4 - \frac{\alpha}{2}\,|v|^2.
\end{equation}

\noindent Here $t \geq 0$ denotes the time variable, $v \in \mathbb{R}^d $ is the velocity variable, $D>0$ is the intensity of the noise,~ $\nabla \cdot$ and ~$\Delta$ respectively denote the divergence and the Laplacian with respect to ~$v$, while ~$\alpha>0$ is a parameter  of the external potential
and ~$\mathbf{u}_{f} $ denotes the average velocity
\begin{align*}
    \mathbf{u}_{f}:= \frac{\int_{\mathbb{R}^d} v \, f(t,v)\,dv}{\int_{\mathbb{R}^d}f(t,v)\,dv}.
\end{align*}
In contrast with the classical Vicsek model, we do not require that ~$|v|$ is constant.
For any ~$\mathbf{u}\in \mathbb{R}^d$, we define a Gibbs state as 
\begin{align}
\label{Gibbsstate}
    G_\mathbf{u}(v):=\mathcal{Z}(\mathbf{u})^{-1} {e^{-\frac{1}{D}\left(\frac{1}{2}|v-\mathbf{u}|^2+ \psi_{\alpha}(v)\right)}}
\end{align}
where ~$\mathcal{Z}(\mathbf{u})$ is a renormalization constant. A stationary state is a Gibbs state with a specific value of the speed ~$|\mathbf{u}|$. In \cite{BarbaroCanizoDegondphasetransi} it was proved  that a phase transition occurs and the following refined result is taken from \cite{Li2021}.

\begin{proposition}[{\rm \cite[Theorem~1.1]{Li2021}}]
\label{prop0}
Let ~$D_*>0$ be the unique positive root of 
\[h(D):=\int_{0}^{\infty}\left(s^{d+1}-s^{d+3} \right)e^{-\frac{\psi_{\alpha}(v)}{D}-\frac{|v|^2}{2D}}\,dv.\] 
For $D<D_*,$ let  $r(D)>0$ be the unique positive root of 
\begin{align}
\label{H}
     {H}(R):= \int_{\mathbb{R}^d}{(v_1-R)\,e^{-\frac{1}{D}\left(\psi_{\alpha}(v)+\frac{|v|^2}{2}-v_1R\right)}}\, dv. 
\end{align}

\begin{itemize}
    \item If $D\geq D_*$, then $G_0$ is the unique stationary state. It is isotropic and stable.
    \item If $D<D_*$, then 
    $G_\mathbf{u}$ is a stationary state for any $u\in \mathcal{S}\cup \{0\}$ where \begin{align}
    \label{S}
        \mathcal{S}:=r(D)\mathbb{S}^{d-1}
    \end{align} and there are no other stationary states. Moreover, $G_0$ is unstable and the set of polarized states $\{G_\mathbf{u}\}_{\mathbf{u} \in \mathcal{S}}$  is stable.
    
\end{itemize}
\end{proposition}

\noindent For two nonnegative functions $f$ and $g$ of unit mass  we define the \emph{relative entropy} as  
\begin{align*}
    \mathcal{H}[f|g]:= \int_{\mathbb{R}^d} f\log\left(\frac{f}{g}\right)dv
\end{align*}
and the \emph{free energy} associated to  \Cref{PDE} as
\begin{align}
\label{freenergy}
    \mathcal{F}[f]= D\int_{\mathbb{R}^d} f \log f \,dv + \int_{\mathbb{R}^d} f \,\psi_{\alpha} \,dv +\frac{1}{2}\int_{\mathbb{R}^d} |v-\mathbf{u}_{f}|^2\,f\,dv.
\end{align}

Our first main result deals with the large time asymptotics of the solution and the identification of a limiting velocity.

\begin{theorem}{}
\label{thm2}
Let $D<D_*$. Let $f$ be a nonnegative solution to \Cref{PDE} with  initial \\ \mbox{datum ~$f_{\mathrm{\mathrm{in}} } \in L^1_+(\mathbb{R}^d)$} such that ~$\mathcal{F}[f_{\mathrm{\mathrm{in}}}]<\mathcal{F}[G_{0}]$ and ~$\| f_{\mathrm{in}}\|_{L^1(\mathbb{R}^d)}=1$. Then
   the average speed $\mathbf{u}_{f}$ converges as $t \rightarrow\infty$ to a unique limit $\mathbf{u}_{\infty} \in \mathcal{S}$ and for any $\beta>0$,
  \begin{align*}
\lim_{t\rightarrow \infty} t^{\beta}\left(\mathcal{H}\big[f(t,\cdot) \mid G_{\mathbf{u}_{\infty}}\big] + |\mathbf{u}_{f}(t) - \mathbf{u}_{\infty}|^2\right) = 0
\end{align*}

\end{theorem}

It was known from \cite{Li2021} that $\lim_{t\rightarrow \infty} \inf_{\mathbf{u}\in \mathcal{S}}{\mathcal{H}\big[ f(t,.)|G_\mathbf{u}\big]}=0$ but the existence of a unique limiting Gibbs state and the convergence of $\mathbf{u}_{f}(t)$ as $t\rightarrow \infty$ are new, thus answering a long standing open question. Under additional conditions on the initial datum, we obtain an exponential rate of convergence.

\begin{theorem}{}
\label{thm3}
There is a positive constant $\beta$ depending only on $\alpha$ and $D$ such that under the assumptions of \Cref{thm2}, if $f_{\mathrm{in}} \in \bigcap_{u \in \mathcal{S}} L^2(\mathbb{R}^d,G_\mathbf{u}^{-1}dv)$, then
     \begin{align*}
         \mathcal{H}\big[f(t,.)|G_{\mathbf{u}_{\infty}}\big] +|\mathbf{u}_{f}-\mathbf{u}_{\infty}|^2 =O(e^{-\beta t}) \text{ as } t\rightarrow +\infty.
     \end{align*}
\end{theorem}
The additional assumption is slightly stronger than asking $f_{\mathrm{in}}$ to belong to $ L^2(\mathbb{R}^d,G_\mathbf{u}^{-1}dv)$ for only one~ $\mathbf{u}$. Indeed, the function 
 $v \mapsto \sqrt{G_\mathbf{u}(v) /\left(1 + |v|^{d+1}\right)^{}}$ belongs to $L^2(\mathbb{R}^d, G_\mathbf{u}^{-1}\, \mathrm{d}v)$ but not to~$ L^2(\mathbb{R}^d, G_w^{-1}\, \mathrm{d}v)$ if $w \neq \mathbf{u}$.
The proofs of \Cref{thm2} and \Cref{thm3} are based on a detailed analysis of the free energy. The exponential convergence is obtained by considering the projection of $\mathbf{u}_{f}(t)$ onto $\mathcal{S}$.

\indent
We refer to \cite{BarbaroCanizoDegondphasetransi} for an introduction to \eqref{PDE}. The Vicsek model \cite{Vicsek1995} is the limit of \eqref{PDE} as $\alpha \rightarrow \infty$ according to ~\cite{bostan2013asymptotic}. Models with a large class of interaction potentials but a strong assumption of uniform convexity were investigated in ~\cite{carrillo2003kinetic}. Their results do not apply to ~\Cref{PDE}. The first large time asymptotic result in the non-critical isotropic case ~$D>D_*$ was established in ~\cite{Li2021}. In this paper, Li proved the convergence to the unique stationary state ~$G_0$ with exponential decay rate. We give a different proof  which includes the study of the critical case ~$D=D_*$ and the more general initial datum of~ \Cref{thm2}.
 We adapt techniques from ~\cite{Frouvelle2021,FrouvelleLIU2012} for the Vicsek model.
 
\indent
This paper is organized as follows. Definitions and preliminary results are collected in~ \Cref{sec:preliminary}. In ~\Cref{sec:sec1}, we prove the convergence of the free energy (\Cref{prop1}). ~\Cref{thm2} and ~\Cref{thm3} are proved respectively in  ~\Cref{sec:sec4} and ~\Cref{sec:: section 5}.

\begin{section}{Preliminary results.}
\label{sec:preliminary}

\begin{subsection}{Existence and uniqueness}
    We provide here some already known results and references. Since \Cref{PDE} has a gradient flow structure, a natural approach is to consider a discrete gradient descent scheme in the Wasserstein space, as introduced by Jordan, Kinderlehrer, and Otto in \cite{JordanOttoKinder}. The drawback of this method is that it does not give uniqueness since the free energy is not convex. The existence is detailed in  Chapter~8 of \cite{Villaniopttranspbook}. The proof only requires an initial datum in $L^1(\mathbb{R}^d, \,dx)$ with finite entropy and second moment. In \cite{bolley2010stochastic}, Bolley, Cañizo and Carrillo proved the existence and uniqueness for a large class of kinetic PDE's which covers \eqref{PDE}. However, their method requires some additional  integrability conditions on the initial datum, see \cite{bolley2010stochastic} Theorem 1.2. In \cite{bolley2010stochastic,Villaniopttranspbook}, the solutions are measure-valued in the sense of distributions at $t=0$, and admit a density with respect to the Lebesgue measure for any $t>0$. Here, we assume that $f_{\mathrm{in}} \in L^1(\mathbb{R}^d)$ is nonnegative with unit mass and has finite free energy and second moment so that a solution exists. We use the same notation as in \cite{Li2021}.
\end{subsection}

\begin{subsection}{Fisher information.}
    The free energy defined by \eqref{freenergy} is a Lyapunov function in the sense that for a regular solution $f$ to \Cref{PDE}, we have:
\begin{align}
\label{fisher}
    \frac{d}{dt} \mathcal{F}[f]=-\,\mathcal{I}[f(t,.)]
\end{align}
where 
\begin{align*}
    \mathcal{I}[f(t,.)]=D^2 \int_{\mathbb{R}^d} \left|\nabla \log\left(\frac{f}{G_{\mathbf{u}_{f}}}\right)\right|^2f \, dv
\end{align*}
is the Fisher information. As a consequence, $t\mapsto \mathcal{F}[f(t,.)]$ is nonincreasing for a solution $f$ of \eqref{PDE}. 
\end{subsection}

\begin{subsection}{The compatibility equation.}
A stationary solution $f$ of \eqref{PDE} satisfies $\mathcal{I}[f]=0$. This \\ implies  $f=G_{\mathbf{u}_{f}}$, that is, $f=G_\mathbf{u}$ with the condition of compatibility
\begin{align}
\label{compatibilityequation}
    \mathbf{H}(\mathbf{u}):=\mathbf{u}-\mathbf{u}_{G_\mathbf{u}}=0.
\end{align}  Conversely, by substituting  $G_\mathbf{u}$ into \Cref{PDE}, one can check that~$G_\mathbf{u}$ is a stationary state \\ if~\hbox{$\mathbf{H}(\mathbf{u})=0$.} 
We can write for $\mathbf{u} \neq 0$ \begin{align}
\label{relationimp}
    \mathbf{H}(\mathbf{u})=-\frac{H(|\mathbf{u}|)}{K(|\mathbf{u}|)} \frac{\mathbf{u}}{|\mathbf{u}|}
\end{align}
where $H$ and $K$ are two real functions defined by \eqref{H} and 
\begin{align}
\label{K}
    {K}(R)= \int_{\mathbb{R}^d}{e^{-\frac{1}{D}\left(\frac{1-\alpha}{2}|v|^2+ \frac{\alpha}{4} |v|^4-v_1R\right)}}\ dv.
\end{align}
where $v_1$ denotes the first coordinate of $v\in (v_1,v_2,...v_d) \in \mathbb{R}^d$.
Hence, the stationary states of \Cref{PDE} are characterized by the zeros of $H$. The proof of \Cref{prop0} amounts to finding the zeros of $H$ and we refer to \cite{Li2021} for more details. 
\end{subsection}

\begin{subsection}{A moment lemma.}
For initial conditions with finite free energy, moments of all orders are finite and  uniformly bounded for any $t>0$.
    \begin{lemma}
    \label{lemmedesmoments}
     Let $f$ be the solution to \Cref{PDE} with initial datum ~$f_{\mathrm{in}} \in L^1_+(\mathbb{R}^d)$ such \mbox{that~$\mathcal{F}[f_{\mathrm{in}}]$} is finite and  $\| f_{\mathrm{in}}\|_{L^1(\mathbb{R}^d)}=1$ . Then for all  $k\in \mathbb{N}$, \begin{align*}
         \sup_{t>0}M_k(t)<+\infty \,\,\,\text{where} \,M_k(t):=\int_{\mathbb{R}^d} |v|^kf(t,v)\,dv.
     \end{align*}

\end{lemma}
\begin{proof}
    We learn from Proposition~4.2 in \cite{Li2021} that:
    \begin{align}
    \label{bddbelow}
        \mathcal{F}[f] \geq -\frac{d}{2}\log(2\pi)\,D + \frac{\alpha}{4} \int_{\mathbb{R}^d}f\,|v|^4\,dv - \frac{D+\alpha}{2} \left( \int_{\mathbb{R}^d}f\,|v|^4\,dv\right)^\frac{1}{2}.
    \end{align}
Moreover, if $f$ solves \eqref{PDE}, then $\mathcal{F}[f(t,.)]\leq \mathcal{F}[f_{\mathrm{in}}] $ for all $t\geq 0$ and, as a consequence, $$\sup_{t>0}M_4(t)<+\infty.$$ Since $\lVert f(t,.)\rVert_{L^1(\mathbb{R}^d)}=1$ for all $t \geq 0$, we also have for $k\in [0,4]$ that $\sup_{t>0}M_k(t)<+\infty$ by Hölder's inequality.  Let $m \geq 5$. Suppose that for $k \leq m-1$, we have $\sup_{t>0}M_k(t)<+\infty$. Differentiating $M_k$, we obtain:
 \begin{align*}
 \frac{\mathrm{d}}{\mathrm{d}t} M_k&= Dk\,(d+k-2)\,M_{k-2} -k \,M_k +k\,\left(\int_{\mathbb{R}^d}|v|^{k-2}vf\, dv\right) \cdot \mathbf{u}_{f}  + \alpha\, k \int_{\mathbb{R}^d} (|v|^k-|v|^{k+2})f\,dv.
 \end{align*}
 Using this relation with $k=m-2$, we obtain that the moment of order $m$ is finite for a.e.~$t>0$ (else the derivative of the moment of order $m-2$ would be $-\infty$). Now using $|v|^k-|v|^{k+2} \leq 0 $ for ~$|v|\geq 1$ and ~$|v|^k-|v|^{k+2} \leq 1$ for ~$|v|\leq 1$, we obtain:
\[
\frac{\mathrm{d}}{\mathrm{d}t} M_k \leq D\,(d+k-2)M_{k-2} -k\, M_k +k \left(\int_{\mathbb{R}^d}|v|^{k-2}vf \,dv\right) \cdot \mathbf{u}_{f}  + \alpha \,k.
\]
 Using this inequality with $k=m$, we obtain that $\sup_{t>0}M_m(t) <+\infty$.
\end{proof}
\end{subsection}

\begin{subsection}{Properties of the free energy}

 The free energy is bounded from below in \( L^1(\mathbb{R}^d, (1 + |v|^4)\, \mathrm{d}v) \), thanks to~\eqref{bddbelow}. We know from Corollary~4.1 of~\cite{Li2021} that:
\[\mathcal{F}_*:=
\inf_{f \in L^1(\mathbb{R}^d, (1 + |v|^4)\, \mathrm{d}v)} \mathcal{F}[f] =
\begin{cases}
\mathcal{F}[G_0] & \text{if } D \geq D_* ,\\
\mathcal{F}[G_\mathbf{u}] \text{ with } |u| = r(D) & \text{if } D < D_*.
\end{cases}
\]
The value of \( \mathcal{F}[G_\mathbf{u}] \) does not depend on \(\text{ \( u \in  \mathcal{S} \) }\) defined by \eqref{S} since the free energy is invariant under rotations. 
Let us state here some useful identities. By definition, the normalized constant of a Gibbs  state is 
\begin{align}
\label{normalization}
    \mathcal{Z}(\mathbf{u}):=\int_{\mathbb{R}^d}e^{-\frac{1}{D}\left(\frac{1}{2}|v-\mathbf{u}|^2+ \psi_{\alpha}(v)\right)}\,dv.
\end{align}
With a direct computation, we obtain for all $\mathbf{u}\in \mathbb{R}^d$,
   \begin{align}
   \label{eqbase}
          \mathcal{F}[f]=\,D\,\mathcal{H}[f|G_\mathbf{u}]-\frac{1}{2}\,|\mathbf{u}-\mathbf{u}_{f}|^2 - D \log(\mathcal{Z}(\mathbf{u}))\,. 
      \end{align}
First using \eqref{eqbase} with $\mathbf{u}=\mathbf{u}_{f}$ and then with $f=G_\mathbf{u}$, we obtain thanks to 
\eqref{compatibilityequation}  that
\begin{align}
\label{equtile}
     \mathcal{F}[f]-\mathcal{F}_*= D \,\mathcal{H}[f|G_{\mathbf{u}_{f}}] + \mathcal{V}(\mathbf{u}_{f})-\mathcal{V}_*
\end{align}
where
\begin{align}
\label{V}
    \mathcal{V}(\mathbf{u}):=-\,D\log(\mathcal{Z}(\mathbf{u}))
\end{align}
is a radial function and $\mathcal{V}_*$ denotes the value of $\mathcal{V}$ on $\mathcal{S}$. This expression of the free energy  is crucial for proving \Cref{prop1}. Taking $\mathbf{u}\in \mathcal{S}$ in \Cref{eqbase}, we obtain
    \begin{align}
    \label{eq22}
          \mathcal{F}[f]-\mathcal{F}_*= D\,\mathcal{H}[f|G_\mathbf{u}]-\frac{1}{2}|\mathbf{u}_{f}-\mathbf{u}|^2.
    \end{align}

\end{subsection}

\subsection{Log-Sobolev and Poincaré inequality}
Let us state two inequalities which generalize the classical Gaussian Poincaré and Log-Sobolev inequalities.  
Since~ $\phi_\mathbf{u}(v):=\frac{1}{2}|v-\mathbf{u}|^2+ \frac{\alpha}{4} |v|^4 -\frac{\alpha}{2}|v|^2$ is a bounded perturbation of a uniformly convex potential, it follows from the \emph{carré du champ} method of ~\cite{bakry1985diffusions} and the Holley-Strook perturbation lemma ~\cite{holley1987logarithmic} that the Log-Sobolev and Poincaré inequalities apply with the measure $d\mu=G_\mathbf{u} \,dv$.
\\
\textbf{Poincaré inequality}: For all~ $\mathbf{u}\in \mathbb{R}^d$ there exists ~$\Lambda_\mathbf{u}>0$ such that for all ~$g\in H^1(\mathbb{R}^d,G_\mathbf{u}\,dv)$ satisfying~$\int_{\mathbb{R}^d}g\,G_\mathbf{u} \,dv=0$, there holds 
\begin{equation}
\label{Poincaréinequality}
\tag{P}
\int_{\mathbb{R}^d} |\nabla g|^2 G_\mathbf{u}\, dv \geq \Lambda_\mathbf{u} \int_{\mathbb{R}^d} g^2 G_\mathbf{u}\, dv.
\end{equation}
Since we will only use the Poincaré inequality with~ $\mathbf{u} \in \mathcal{S}$ we will denote by~ $\Lambda(D)$ the common value of ~$\Lambda_\mathbf{u}$ for all ~$\mathbf{u} \in \mathcal{S}$.\\
\textbf{Log-Sobolev inequality}: For all~ $\mathbf{u} \in \mathbb{R}^d$, there exists ~$\mathcal{K}_\mathbf{u}>0$ such that for all nonnegative function of unit mass ~$f$ that satisfy~$\sqrt{fG_\mathbf{u}^{-1}}\in H^1(\mathbb{R}^d,G_\mathbf{u}dv)$ there holds
\begin{equation}
\label{log_sob}
\tag{LSI}
\int_{\mathbb{R}^d} \left|\nabla \log\left( \frac{f}{G_\mathbf{u}}\right)\right|^2 f\, dv \geq \mathcal{K}_\mathbf{u} \int_{\mathbb{R}^d} \log \left(\frac{f}{G_\mathbf{u}}\right) f\, dv.
\end{equation}
In the following, we will use the Log-Sobolev inequality with ~$\mathbf{u}=\mathbf{u}_{f}$ in a neighborhood of~ $\mathcal{S}$. \\
Let  ~$\mathcal{K}_{\epsilon}=\inf_{\mathbf{u} \in \mathbb{R}^d\,: \,\mathrm{dist}(\mathbf{u},\mathcal{S})\leq \epsilon} \mathcal{K}_{\mathbf{u}}>0$ where $\mathrm{dist}(\mathbf{u},\mathcal{S})=\min_{v\in \mathcal{S}}|\mathbf{u}-v|$.

\begin{subsection}{Convergence of the norm of the average speed.}
\label{subsec:convavspeed}
The following result is taken from \cite{Li2021}. 
    \begin{lemma}
\label{lemmejfvois} Let $D>0$.
     Let $f$ be the solution to \Cref{PDE} with initial datum ~$f_{\mathrm{in}} \in L^1_+(\mathbb{R}^d)$ such \mbox{that~$\mathcal{F}[f_{\mathrm{in}}]<+\infty$} and  $\| f_{\mathrm{in}}\|_{L^1(\mathbb{R}^d)}=1$. 
     \\
\begin{itemize}
    \item      If $D\geq D_*$ we have
 \begin{align*}
         \lim_{t \rightarrow{\infty}}\mathbf{u}_{f}(t)=0\,.
         \end{align*}

\item If $D<D_*$ and $\mathcal{F}[f_{\mathrm{in}}]<\mathcal{F}[G_{0}]$ we have
     \begin{align*}
         \lim_{t \rightarrow{\infty}} \mathrm{dist}(\mathbf{u}_{f}(t),\mathcal{S})=0\,.
     \end{align*}
     \end{itemize}
\end{lemma}
For completeness, let us give a proof with some additional details compared to \cite{Li2021}.
\begin{proof}
    First, we prove that for any $D>0$:
      $$\lim_{t\rightarrow \infty}|\mathbf{u}_{f}(t)-\mathbf{u}_{G_{\mathbf{u}_{f}}}(t)| =0.$$
 Let $\gamma(t)= |\mathbf{u}_{f}(t)-\mathbf{u}_{G_{\mathbf{u}_{f}}}(t)|^{\beta}$ with $\beta>0$.
    \begin{align*}
        \int_{0}^{\infty} \gamma(t)dt &= \int_{0}^{\infty} \left| \int_{\mathbb{R}^d}(vf -v\,G_{\mathbf{u}_{f}})\,dv \right|^{\beta}dt \\& \leq\int_{0}^{\infty} \left[  \left(\int_{\mathbb{R}^d} \,|f-G_{\mathbf{u}_{f}}|\,dv \right)^{\frac{3}{4}} \left(\int_{\mathbb{R}^d} |v|^4\,|f+G_{\mathbf{u}_{f}}|\,dv \right)^{\frac{1}{4}}\right]^{\beta}dt \\&\leq C \int_{0}^{\infty}  \left(\int_{\mathbb{R}^d}\, |f-G_{\mathbf{u}_{f}}|\,dv \right)^{\frac{3\beta}{4}}dt
    \end{align*}
using that $M_4$ is bounded by \Cref{lemmedesmoments}.
Choosing $\beta = 8/3$, we obtain
\begin{equation}
\label{eq0}
     \int_{0}^{\infty} \gamma(t)dt \leq  C \int_{0}^{\infty}  \left(\int_{\mathbb{R}^d} |f-G_{\mathbf{u}_{f}}|dv \right)^{2}.
\end{equation}
By the Csiszár-Kulback inequality, we know that 
\begin{align}
\tag{C-K}\label{eq100}
    \frac{1}{2} \left\| f - G_{\mathbf{u}_{f}} \right\|^2_{L^1(\mathbb{R}^d)} \leq \mathcal{H}[f | G_{\mathbf{u}_{f}}]
\end{align}
Using  \eqref{log_sob} and \eqref{fisher}, we obtain
\begin{align*}
    \frac{d}{dt} \mathcal{F}[f(t,.)]=-\mathcal{I}[f(t,.)]\leq -\,D^2 \,\mathcal{K}\, \mathcal{H}[f|G_{\mathbf{u}_{f}}],
\end{align*}
where $\mathcal{K}=\mathcal{K}_{\epsilon}$ is the positive constant in \eqref{log_sob} with $\epsilon=\sup_{t \geq 0}\mathrm{dist}(\mathbf{u}_{f}(t),\mathcal{S})<+\infty$ by \Cref{lemmedesmoments}. 
Since $t \mapsto \mathcal{F}[f(t,.)]$ is nonincreasing and bounded from below, it has a limit $\mathcal{F}_{\infty} \in \mathbb{R}$ as $t \rightarrow+\infty$. Integrating the previous inequality in time, we get a de Bruijn type estimate
\begin{align}
\label{eqq2}
    \int_{0}^{\infty} \mathcal{H}[f|G_{\mathbf{u}_{f}}] dt \leq \frac{1}{\mathcal{K}D^2}( \mathcal{F}[f_{\mathrm{in}}]- \mathcal{F}_{\infty}).
\end{align}
With \eqref{eq0}, \eqref{eq100} and \eqref{eqq2}, we obtain that  $\gamma(t)= W(\mathbf{u}_{f}(t))$ is integrable on $\mathbb{R}_+$, where
$W$ is defined on $\mathbb{R}^d$ by \mbox{$W(Y)= |Y-\mathbf{u}_{G_Y}|^{\beta}$}. Observe that  $\gamma'(t)= \nabla W(\mathbf{u}_{f}(t))  \mathbf{u}_{f}'(t)$ is bounded since $\mathbf{u}_{f} $ and~\mbox{$\mathbf{u}_{f}'=\alpha \mathbf{u}_{f}-\alpha\int_{\mathbb{R}^d}v|v|^2f\,dv$}
are uniformly bounded in $t\geq0$. Since $\gamma$ is integrable in time and uniformly continuous, we obtain $\lim_{t\rightarrow \infty} \gamma(t)=0$ thanks to Barbalat's lemma. Using \eqref{relationimp}, we obtain\,: 
\begin{align*}
    |\mathbf{u}_{f}(t)-\mathbf{u}_{G_{\mathbf{u}_{f}}}(t)| = -\,\frac{{H}(|\mathbf{u}_{f}|)}{{K}(|\mathbf{u}_{f}|)} \frac{\mathbf{u}_{f}}{|\mathbf{u}_{f}|}.
\end{align*}
We conclude that 
\begin{align*}
    \lim_{t\rightarrow \infty}{H}(|\mathbf{u}_{f}(t)|) =0\,.
\end{align*}
If $D\geq D_*$ then 0 is the unique zero of $H$ (see \cite{Li2021} for the proof), and 
\begin{align*}
    \lim_{t \rightarrow{\infty}}\mathbf{u}_{f}(t)=0\,.
\end{align*}
If $D<D_*$, we recall that the zeros of ${H}$ on $\mathbb{R}_+$ are exactly $r(D)$ and $0$ by \Cref{prop0}. If $\mathcal{F}[f_{\mathrm{in}}]<\mathcal{F}[G_0]$, the case $\lim_{t \rightarrow{\infty}}\mathbf{u}_{f}(t) =0$ is excluded, since in this case \[\mathcal{F}[f_{\mathrm{in}}]\geq \lim_{t\rightarrow \infty}\mathcal{F}[f]\geq \inf_{\mathbf{u}_{f}=0}\mathcal{F}[f]=\mathcal{F}[G_0]>\mathcal{F}[f_{\mathrm{in}}]. \]
We conclude that
\begin{equation*}
    \begin{aligned}
        \lim_{t \rightarrow{\infty}}\mathrm{dist}(\mathbf{u}_{f},\mathcal{S}) = 0\,. 
    \end{aligned}
\end{equation*}
\end{proof}
\end{subsection}
\end{section}

\begin{section}{A rate of decay for the free energy.}

\label{sec:sec1}
The aim of this section is to prove the convergence of the free energy to its minimum. The proof relies on a differential inequality on ~$\mathcal{F}[f]$ based on ~ \eqref{equtile}, \eqref{log_sob} and Polyak–Łojasiewicz type estimates ~\eqref{P_L_M_B}. It also gives an algebraic rate of convergence for the decay of the free energy.

\begin{subsection}{A detailed analysis of  \texorpdfstring{$\mathcal{V}$}{V}.}
    
Let us recall that the function $\mathcal{V}$ is defined through \eqref{V} and \eqref{normalization} as 
\begin{align}
\label{V2}
     \mathcal{V}(\mathbf{u})=-\,D\log\left( \int_{\mathbb{R}^d}e^{-\frac{1}{D}\left(\frac{1}{2}|v-\mathbf{u}|^2+ \psi_{\alpha}(v)\right)}\,dv\right).
\end{align}
In particular, $\mathcal{V}$ is smooth and radial and with a slight abuse of notation, we shall write $\mathcal{V}(\mathbf{u})=\mathcal{V}(|\mathbf{u}|)$. The key point is to prove estimates of Polyak–Łojasiewicz type, which are based on the following Taylor expansion.

\begin{lemma}[]
\label{P-L}
With $\mathcal{V}$ defined by \eqref{V2}, we have the following estimates.
\leavevmode\\
\begin{itemize}
    \item If \( D < D_* \). For all $\epsilon >0$,  there exist $\mu_1>0$ and $\mu_2>0$ which satisfy $|\mu_i-\mathcal{V}''(r(D))|<\epsilon$ for $i=1,2$ and $0<\delta< \mathrm{min}\big\{r(D)/2, \epsilon\big \}$ such that for all $\mathbf{u}$ satisfying $\mathrm{dist}(\mathbf{u},\mathcal{S})<\delta $ we have,
    \begin{align}\label{P_L_M_B}
        \frac{\mu_1}{2} \mathrm{dist}(\mathbf{u},\mathcal{S})^2 
        \leq \mathcal{V}(\mathbf{u}) - \mathcal{V}_* 
        \leq \frac{1}{2\mu_2} |\nabla \mathcal{V}(\mathbf{u})|^2 .
    \end{align}
    \item If \( D > D_* \). For all $\epsilon >0$,  there exist $\mu_1>0$ and $\mu_2>0$ which satisfy $|\mu_i-\mathcal{V}''(0)|<\epsilon$ for $i=1,2$ and $0<\delta<\epsilon$ such that for all $|\mathbf{u}|<\delta$ we have,
    \begin{align}
    \label{Pol_D>D_*}
    \frac{\mu_1}{2} |\mathbf{u}|^2\leq \mathcal{V}(\mathbf{u})-\mathcal{V}(0)\leq \frac{1}{2\mu_2}|\nabla\mathcal{V}(\mathbf{u})|^{2} .
\end{align}
\item If \( D = D_* \). For all $\epsilon >0$,  there exist $\mu_1>0$ and $\mu_2>0$ which satisfy $|\mu_i-\mathcal{V}^{(4)}(0)|<\epsilon$ for $i=1,2$ and  $0<\delta<\epsilon$ such that for all $|\mathbf{u}|<\delta$ we have,
\begin{align}
\label{polD=D_*}
    \frac{\mu_1}{24} |\mathbf{u}|^4\leq \mathcal{V}(\mathbf{u})-\mathcal{V}(0)\leq \ \frac{6^{\frac{4}{3}}}{24 \mu_2^{\frac{1}{3}}} |\nabla\mathcal{V}(\mathbf{u})|^{\frac{4}{3}}.
\end{align}
\end{itemize}
\end{lemma}
In the following, we will denote by 
\begin{equation}
\label{N}
\mathcal{N}:=
\left\{
\begin{array}{ll}
\{\mathbf{u} \in \mathbb{R}^d \text{ such that } \mathrm{dist}(\mathbf{u},\mathcal{S}) < \delta\} & \text{if } D < D_* \\
B(0,\delta) & \text{if } D \geq D_*
\end{array} 
\right.
\quad \text{where } \delta \text{ is defined as in \Cref{P-L}.}
\end{equation}

\begin{proof}
    The proof follows from a Taylor expansion of $\mathcal{V}$ and $\mathcal{V}'$ around $\mathbf{u}\in \mathcal{S}$ if $D<D_*$ and around $0$ if $D\geq D_*$. Since $\mathcal{S}$ and $0$ are critical points of $\mathcal{V}$ (see \Cref{lemmappendic1}), one has to check that:
    \begin{itemize}
        \item $\mathcal{V}''(r(D))>0$ if $D<D_*$,
        \item $\mathcal{V}''(0)>0 $ if $D>D_*$,
        \item $\mathcal{V}''(0)=\mathcal{V}^{(3)}(0)=0$ and $\mathcal{V}^{(4)}(0)>0$ if $D=D_*$. \qedhere
    \end{itemize}
\end{proof}
We prove these three points in the following lemmas by studying the functions $H$ and $K$ introduced in  \cite{Li2021} defined by \eqref{H} and \eqref{K},
\[K(r)= \int_{\mathbb{R}^d}{e^{-\frac{1}{D}\left(\frac{1-\alpha}{2}|v|^2+ \frac{\alpha}{4} |v|^4-v_1r\right)}}dv \text{ and } H(r)= \int_{\mathbb{R}^d}{(v_1-r)e^{-\frac{1}{D}\left(\frac{1-\alpha}{2}|v|^2+ \frac{\alpha}{4} |v|^4-v_1r\right)}}dv\]
and 
 \[\mathcal{Z}(r)= \int_{\mathbb{R}^d}{e^{-\frac{1}{D}(\frac{1}{2}|v-re_1|^2+ \frac{\alpha}{4} |v|^4 -\frac{\alpha}{2}|v|^2)}dv} \text{ according to } \eqref{normalization}.\]
 \Cref{lemmappendic1} is a reformulation of results coming from \cite{Li2021} for the function $\mathcal{V}$.

\begin{lemma}
\label{lemmappendic1}
    Let $D>0$, then the set of minimizers of $\mathcal{V}$ is $\mathcal{S}$ if $D<D_*$ and $\{0\} $ if $D\geq D_*$.
\end{lemma}
\begin{proof}
Notice that
$\mathcal{Z}(r)= e^{-\frac{r^2}{2D}} K(r)$. 
By differentiating  \eqref{V} and using  $DK'(r)-rK(r)=H(r)$, we obtain
\begin{align*}
    \mathcal{V}'(r)
    &= - \frac{H(r)}{ K(r)}.
\end{align*}
 It is shown in the proof of \cite[Proposition~2.3]{Li2021}
  for the case $d \geq  2$, and in the proof of \cite[Proposition~2.2]{Li2021} for $d = 1$, that
\begin{align*}
    \lim_{r \rightarrow \infty} H(r) = -\infty.
\end{align*}
Let $D\geq D_*$
Since the only zero of $H$ is $0$ by \Cref{prop0} and $H$ is continuous, we conclude that~ $H(r)<0$ for all $r>0$ \emph{i.e.} $\mathcal{V}'(r) > 0$ for all $r > 0$ so $0$ is the unique minimum of $\mathcal{V}$.
\\
Otherwise, if $0<D<D_*$, the function $H$ has a unique zero on $\mathbb{R}_+^*$  denoted by $r(D)$ by \Cref{prop0}. Moreover, for $r \in (0,r(D))$, we have $H(r)>0$, and for $r>r(D)$, we have $H(r)<0$ according to the proof of \cite[Proposition~2.2, Proposition~2.3]{Li2021}. Since $K(r)$ is positive, this concludes the lemma. 
\end{proof}
Taking another derivative and since $H(r(D))=0$,
 we obtain:
 \begin{align}
 \label{lasteq}
     {\mathcal{V}}''(r(D))=- \frac{H'(r(D))}{K(r(D))}.
 \end{align}

\begin{lemma}
\label{Hprime}
    Let  $D<D_*$, then $\mathcal{V}''(r(D))>0.$
\end{lemma}

\begin{proof}We will prove that $H'$ has a unique zero $\mathbf{u}_0$ and $0<\mathbf{u}_0<r(D)$.
    Since $D<D_{*}$, thanks to \Cref{prop0}, we know that $H(0)=0$ and   $H(r(D))=0$ with $r(D)>0$. Using Rolle's lemma, we conclude that $H'$ has at least one zero on $(0,r(D))$.
    \\
Concerning uniqueness,
let $r_1 \in (0,r(D))$ be a zero of $H'$ and $r_2>r_1.$ 
In the following, we use the notation~\mbox{$\phi_{\alpha}(s):=\frac{\alpha}{4}s^4+\frac{1-\alpha}{2}s^2=\psi_{\alpha}(s)+\frac{s^2}{2}$} coming from \cite{Li2021}.
Thanks to \cite[Lemma~2.1 and \mbox{Section 2.3}]{Li2021},   we can write,

\begin{align} \label{eqpolaire} H(r) = \alpha |\mathbb{S}^{d-2}| \int_{0}^{\infty} (1 - s^2) s^d e^{- \frac{\phi_{\alpha}(s)}{D}} h\left(\frac{rs}{D}\right) ds \quad \text{where } h(s) = \begin{cases}
  \int_{0}^{\pi} \cos\theta (\sin\theta)^{d-2} e^{s \cos\theta} d\theta & \text{if } d \geq 2 \\
  \sinh(s) & \text{if } d=1,
\end{cases}
\end{align}
with the convention $|\mathbb{S}^0|=|\mathbb{S}^{-1}|=2$. If $d\geq 2$, we deduce from 
\begin{align*}
    c(s)&:=\frac{h''(s)}{h'(s)}
    = \frac{\int_{0}^{\pi}  (\cos\theta)^3 (\sin \theta )^{d-2} e^{ s \cos \theta} d \theta}{\int_{0}^{\pi}  (\cos\theta)^2 (\sin \theta )^{d-2} e^{ s \cos \theta} d \theta}
\end{align*}
that,
\begin{align*}
    c'(s)&= \frac{\int_{0}^{\pi}  (\cos\theta)^4 (\sin \theta )^{d-2} e^{ s \cos \theta} d \theta}{\int_{0}^{\pi}  (\cos\theta)^2 (\sin \theta )^{d-2} e^{ s \cos \theta} d \theta} - \left(\frac{\int_{0}^{\pi}  (\cos\theta)^3 (\sin \theta )^{d-2} e^{ s \cos \theta} d \theta}{\int_{0}^{\pi}  (\cos\theta)^2 (\sin \theta )^{d-2} e^{ s \cos \theta} d \theta}\right)^2 \geq 0,
\end{align*}
where we used the Cauchy-Schwarz inequality with the measure $d\mu= (\cos\theta)^2\, (\sin \theta )^{d-2} \,e^{ s \cos \theta} \,d \theta$. \\
If $d=1$, it is straightforward that $c(s)=h''(s)/h'(s)$ is non-decreasing.
For all $d\geq 1$,
let $k$ be the function defined by~ \mbox{$k(s)= {h'(\frac{r_2 s}{D})}/{ h'(\frac{r_1 s}{D})}$}. Differentiating  $k$ and using that $s\mapsto s\,c(s)$ is monotone increasing, we obtain that $k$ is increasing on $(0,+\infty)$. Finally we have for $d\geq 1$,

\begin{align*}
\int_{0}^1 (1-s^2)s^{d+1} e^{-\frac{\phi_{\alpha}(s)}{D}} h'\left(\frac{r_2 s}{D}\right)ds 
&= \int_{0}^1 (1-s^2)s^{d+1} e^{-\frac{\phi_{\alpha}(s)}{D}} k(s)h'\left(\frac{r_1 s}{D}\right)ds \\
&< \int_{0}^1 (1-s^2)s^{d+1} e^{-\frac{\phi_{\alpha}(s)}{D}} k(1)h'\left(\frac{r_1 s}{D}\right)ds \\
&\hspace{1cm}= \int_{1}^{\infty} (s^2-1)s^{d+1} e^{-\frac{\phi_{\alpha}(s)}{D}} k(1)h'\left(\frac{r_1 s}{D}\right)ds \\
&\hspace{1cm}< \int_{1}^{\infty} (s^2-1)s^{d+1} e^{-\frac{\phi_{\alpha}(s)}{D}} k(s)h'\left(\frac{r_1 s}{D}\right)ds \\
&\hspace{2cm}= \int_{1}^{\infty} (s^2-1)s^{d+1} e^{-\frac{\phi_{\alpha}(s)}{D}} h'\left(\frac{r_2 s}{D}\right)ds,
\end{align*}
that is, $H'(r_2)<0.$
In the end, for $d\geq 1$ we obtain $H'(r(D))<0 $, \emph{i.e.}, ~$\mathcal{V}''(r(D))>0$.
\end{proof}

\begin{lemma}
\label{lemmeappendice3}
    Let $D>D_*$, then $\mathcal{V}''(0)> 0$. If $D=D_*$, then  $\mathcal{V}''(0)=\mathcal{V}^{(3)}(0)=0$ and $\mathcal{V}^{(4)}(0) > 0$.
\end{lemma}
\begin{proof}
If $D>D_*$, $\mathcal{V}''(0)>0$ follows from \eqref{lasteq} and  $H'(0)<0$. For $D=D_*$, we notice that $H'(0)=0$ by \cite[Proposition~2.1]{Li2021}, that is, $\mathcal{V}''(0)=0$.
 Since $0$ is a minimum of $\mathcal{V}$, we conclude that $\mathcal{V}^{(3)}(0)=0$. Since  
 \begin{align*}
     \mathcal{V}^{(4)}(0)= -\frac{H^{(3)}(0)}{K(0)}\geq 0
 \end{align*}
   because  $H^{(3)}(0) \leq 0$ and $H'(r)<0$ for $r>0$ by  \cite[Proposition~2.3]{Li2021}, it is enough to prove that $H^{(3)}(0) \neq 0$. 
For $d=1$, it is proved in  \cite[Lemma~2.1]{Li2021}. For $d\geq 2$,
using the relation \eqref{eqpolaire} and differentiating $H$ with respect to $r$, we obtain:
\begin{itemize}
    \item $H'(0)= \frac{\alpha}{D_*} |\mathbb{S}^{d-2}|\int_{0}^{\infty}(1-s^2)\,s^{d+1} \,e^{- \frac{\phi_{\alpha}(s)}{D_*}} \, ds$,
    \item $H^{(3)}(0)=\frac{\alpha}{D_*^3} |\mathbb{S}^{d-2}|\int_{0}^{\infty}(1-s^2)\,s^{d+3}\, e^{- \frac{\phi_{\alpha}(s)}{D_*}}  ds$.
\end{itemize}
The function 
\begin{align*}
    \varphi(t):=\int_{0}^{\infty} s^t\, e^{- \frac{\phi_{\alpha}(s)}{D_*}}  \,ds ,
\end{align*}
 is strictly convex because 
 \begin{align*}
     \varphi''(t)= \int_{0}^{\infty} \log(s)^2\,s^t \,e^{- \frac{\phi_{\alpha}(s)}{D_*}}  \,ds>0.
 \end{align*}
  We know from \cite[Proposition 2.1]{Li2021}  that $H'(0)=0$, that is $\varphi(d+1)= \varphi(d+3)$ and since $\varphi$ is strictly convex, we can't have $\varphi(d+1)=\varphi(d+3)=\varphi(d+5)$ that is $H^{(3)}(0) \neq 0$.
\end{proof}

\end{subsection}
\begin{subsection}{Case \texorpdfstring{$D<D_*$}{DD*}}
We recall that $\mathcal{F}_*=\inf{\mathcal{F}[f]}$.
\label{subsec:3.1}
\begin{proposition}
\label{prop1}
Under the assumptions of \Cref{thm2}, for all  $\beta>0$, 
    \begin{align*}
        \lim_{t\rightarrow+\infty}t^{\beta}\left(\mathcal{F}[f(t,.)]-\mathcal{F}_*\right) =0.
        \end{align*}

\end{proposition}

\begin{proof}
Let $\epsilon>0$. By \Cref{P-L} there exist positive constants $\mu_1$ and $\mu_2$ such that $|\mu_i-\mathcal{V}''(r(D))|<\epsilon$ and a neighborhood $\mathcal{N}=\mathcal{N}_{\epsilon}$ of $\mathcal{S}$ such that
\eqref{P_L_M_B} holds for all $\mathbf{u}\in \mathcal{N}$.
By \Cref{lemmejfvois}, that is $\lim_{t\rightarrow \infty}\mathrm{dist}(\mathbf{u}_{f}(t),\mathcal{S})=0$, there exists $t_0>0$ such that for all $t\geq t_0$ we have $\mathbf{u}_{f}(t) \in \mathcal{N}$. Hence, we have,
\begin{align}
\label{domV}
    \mathcal{V}(\mathbf{u}_{f}(t))-\mathcal{V}_* \leq \frac{1}{2\mu_2}|\nabla\mathcal{V}(\mathbf{u}_{f}(t))|^2 &=\frac{1}{2\mu_2}|\mathbf{u}_{f}-\mathbf{u}_{G_{\mathbf{u}_{f}}}|^2 \notag \\
    &\leq \frac{1}{2\mu_2} \left(\int_{\mathbb{R}^d} |v|^q\,|f+G_{\mathbf{u}_{f}}|dv \right)^{\frac{2}{q}}\vert \vert f-G_{\mathbf{u}_{f}}\vert \vert _{L^1(\mathbb{R}^d)}^{\frac{2}{p}} \notag\\
    & \leq A_q (2\mathcal{H}[f(t,.)|G_{\mathbf{u}_{f}(t)}])^{\frac{1}{p}} \text{ for all } t\geq t_0,
\end{align}
using \eqref{eq100}.
Here $p$ and $q$ are  Hölder conjugate exponents in $(1,+\infty)$ and  $$A_q= \frac{1}{2\mu_2}\sup_{t>t_0}{\left(\int_{\mathbb{R}^d} |v|^q|f+G_{\mathbf{u}_{f}}|\,dv \right)^{\frac{2}{q}}} $$  is finite according to \Cref{lemmedesmoments}. Using  \eqref{domV} and \eqref{equtile}, we obtain 
\begin{align}
\label{eqbof}
    \mathcal{F}[f]-\mathcal{F}_* \leq  \varphi(\mathcal{H}[f|G_{\mathbf{u}_{f}}]) \,\,\,\forall
t\geq t_0
\end{align}
where $\varphi(x)=D\,x + 2^{\frac{1}{p}}A_q \,x^{\frac{1}{p}}$ is monotone increasing. Hence,
\begin{align}
\label{gronwall}
    \frac{d}{dt} (\mathcal{F}[f(t,.)]-\mathcal{F}_*)=-\mathcal{I}[f(t,.)] 
    \leq  -D^2 \mathcal{K} \mathcal{H}[f(t,.)|G_{\mathbf{u}_{f}(t)}]
    \leq -D^2 \mathcal{K} \varphi^{-1}\left(\mathcal{F}[f(t,.)]-\mathcal{F}_*\right),
\end{align}
using \eqref{log_sob} with constant $\mathcal{K}=\mathcal{K}_{\epsilon}$ since $\mathrm{dist}(\mathbf{u}_{f}(t),\mathcal{S})<\epsilon$ for all $t \geq t_0$.
Since $\mathcal{F}[f(t,.)]-\mathcal{F}_* \geq 0$ and ~$\mathcal{F}[f]$ is nonincreasing, this proves that
\begin{align}
\label{limF}
    \lim_{t \rightarrow \infty}\mathcal{F}[f(t,.)]=\mathcal{F}_*.
\end{align}

Using \eqref{gronwall} and \eqref{limF}, we obtain  that there exists $t_1\geq t_0$
\begin{align}
\label{ineqadif}
    \frac{d}{dt} (\mathcal{F}[f(t,.)]-\mathcal{F}_*) \leq  -C(\mathcal{F}[f(t,.)]-\mathcal{F}_*)^{p} \, \,\,\,\forall t \geq t_1,
\end{align}
for some constant $C>0$. Integrating \eqref{ineqadif} with respect to $t\geq t_1$, we obtain the result with $\beta<1/(p-1)$.

\end{proof}

\end{subsection}
\begin{subsection}{Case \texorpdfstring{$D\geq D_*$}{DD*}}
\label{sec: D D*}
In \cite{Li2021}, the exponential decay of $f $ to $G_0$ with respect to the $L^2(\mathbb{R}^d,G_0^{-1}dv)$ is proved for $D>D_*$ and for initial data $f_{\mathrm{in}} \in L^2(\mathbb{R}^d,G_0^{-1}dv)$. In this section, following the same method as in 
\Cref{subsec:3.1}, we obtain that the convergence of $f$ to $G_0$ (in relative entropy) is still true by only assuming $f_{\mathrm{in}} \in L^1(\mathbb{R}^d)$ and $\mathcal{F}[f_{\mathrm{in}}]<+ \infty$. We only get an arbitrary algebraic rate of convergence. The limit case $D=D_*$ is also of interest. We know that for $D=D_*$, \Cref{PDE} has a unique stationary state $G_0$. To our knowledge, there is no convergence result in the literature even under the  condition  $f_{\mathrm{in}} \in L^2(G_0^{-1})$. We are going to prove this convergence in relative entropy.

\begin{proposition}
\label{prop2}
Let $D\geq D_*$ and
    let $f$ be the solution to the \Cref{PDE} with initial datum \mbox{~$f_{\mathrm{in}} \in L^1_+(\mathbb{R}^d)$}  such that $\mathcal{F}[f_{\mathrm{ini
    }}]<+\infty$ and $\|f_{\mathrm{in}} \|_{L_1(\mathbb{R}^d)}=1$.
    For all $\beta >0$ if $D>D_*$ and all $\beta \in (0,2)$ if $D=D_*$, then
    \begin{align*}
        \lim_{t\rightarrow \infty }t^{\beta}\left(\mathcal{F}[f(t,.)]-\mathcal{F}[G_0]\right)=0.
    \end{align*}
\end{proposition}
\begin{proof}
    The proof goes along the same steps as the proof of \Cref{prop1}:
    \begin{enumerate}
    \item A log-Sobolev inequality holds for $G_\mathbf{u}$ where $\mathbf{u} $ belongs to a neighborhood of $0$ and it is also the case since $-\log(G_0)$ is a bounded perturbation of a uniformly convex potential.
    \item $\mathbf{u}_{f} $ converges to $0$ according to the case  $D\geq D_*$ of \Cref{lemmejfvois}.
    \item Estimates \eqref{Pol_D>D_*} and \eqref{polD=D_*} apply. 
\end{enumerate}
Doing the same computations as in \Cref{prop1}, we conclude with $\beta< 1/(p-1)$ for $D>D_*$ and with $\beta< 2/(3p-2) $ for $D=D_*$.
\end{proof}

\begin{corollary}
    Under the assumptions of \Cref{prop2}, for all $\beta>0$ if $D>D_*$ and all $\beta \in (0,1)$ if $D=D_*$ then, 
    \begin{align*}
        \lim_{t\rightarrow \infty}t^{\beta} \mathcal{H}[f(t,.)|G_0]=0.
    \end{align*}
\end{corollary}
\begin{proof}
 By \Cref{lemmejfvois}, we know that $\lim_{t \rightarrow \infty}\mathbf{u}_{f}(t)=0$. Using a Taylor expansion near $0$, using \Cref{lemmeappendice3}, we obtain that there exist positive constants $C_1$ and $C_2$ such that for $t $ large enough, 
  \begin{align*}
      |\mathbf{u}_{f}|^2\leq  C_1\left(\mathcal{V}(\mathbf{u}_{f})-\mathcal{V}(0) \right)\leq C_1\left(\mathcal{F}(f)-\mathcal{F}(G_0) \right)  \text{ if } D>D_*,
  \end{align*} 
\begin{align*}
    |\mathbf{u}_{f}|^4 \leq  C_2\left(\mathcal{V}(\mathbf{u}_{f})-\mathcal{V}(0) \right)\leq C_2\left(\mathcal{F}(f)-\mathcal{F}(G_0) \right) \text{ if } D=D_*,
\end{align*}
where we used the inequality $\mathcal{V}(\mathbf{u}_{f})-\mathcal{V}(0)\leq \mathcal{F}(f)-\mathcal{F}(G_0)$ which is a consequence of \eqref{equtile}.
This concludes the proof using \Cref{prop2} and \eqref{eq22}.
\end{proof}

\end{subsection}

\begin{section}{Convergence of \texorpdfstring{$\mathbf{u}_{f}$}{uf} and proof of~\texorpdfstring{\Cref{thm2}}{thm2}}

\label{sec:sec4}
Since we know from \Cref{prop1} and \Cref{lemmejfvois} that~$ \lim_{t\rightarrow \infty}\mathcal{F}[f]=\mathcal{F}_*$  and  \\ \hbox{~$ \lim_{t \rightarrow \infty}{\mathrm{dist}(\mathbf{u}_{f}, \mathcal{S})}=0$}, we deduce from \eqref{eq22} that $\lim_{t \rightarrow \infty}\inf_{\mathbf{u}\in \mathcal{S}}\mathcal{H}[f(t,.)|G_\mathbf{u}]=0$. We now want to prove that $f$ converges to a unique stationary solution (in relative entropy). Using \eqref{eq22} this is equivalent to proving that $\mathbf{u}_{f}$ converges.
 We proved in \Cref{subsec:convavspeed} that $\mathbf{u}_{f}$ converges to $\mathcal{S}$, which means that $\lim_{t\rightarrow\infty}|\mathbf{u}_{f}(t)|=r(D), $ but we did not prove yet that $\mathbf{u}_{f}(t)$ converges as $t \rightarrow + \infty$ to some fixed $\mathbf{u}\in \mathcal{S}$.

    \begin{proposition}
    \label{convdeJf}
       Under the assumptions of \Cref{thm2},  there exists $\mathbf{u}_{\infty} \in \mathcal{S}$ such that for all $\beta >0$, 
        \begin{align*}
            \lim_{t\rightarrow \infty}t^{\beta}|\mathbf{u}_{f}(t)- \mathbf{u}_{\infty}| =0.
        \end{align*}
    \end{proposition}
\begin{proof}
    Using
    \begin{align}
    \label{derivéuf}
        \frac{d \mathbf{u}_{f}}{dt}=\alpha \mathbf{u}_{f}- \alpha \int_{\mathbb{R}^d}v|v|^2fdv 
    \end{align}
    and that for all $\mathbf{u} \in \mathcal{S}$ we have $\mathbf{u}= \int_{\mathbb{R}^d}|v|^2\,v G_{\mathbf{u}}\,dv$, we obtain
    \begin{align*}
     \left|  \frac{d\mathbf{u}_{f} }{dt}(t)\right|& \leq  \alpha \left| \mathbf{u}_{f} -\mathbf{u}_{} \right| + \alpha \int_{\mathbb{R}^d}  |v|^3 \, | f-G_{\mathbf{u}_{}}|\, dv \\
     &\leq  \alpha \left| \mathbf{u}_{f} -\mathbf{u}_{} \right| + \alpha \left( \int_{\mathbb{R}^d} \, |v|^{3q} | f-G_{\mathbf{u}_{}}| \,dv \right)^{\frac{1}{q}} \left( \int_{\mathbb{R}^d} | f-G_{\mathbf{u}_{}}| dv \,  \right)^{\frac{1}{p}} \\
     & \leq  \alpha \left| \mathbf{u}_{f}(t) -\mathbf{u}_{} \right| + C_q \,\lVert f(t,.)-G_{\mathbf{u}_{}}\rVert_{L^1(\mathbb{R}^d)}^{\frac{1}{p}}.
\end{align*}
Here, $p$ and $q$ are Hölder conjugate exponents in $(1,+\infty)$ and $$C_q = \sup_{ t >t_0}{ \alpha \left( \int_{\mathbb{R}^d}  |v|^{3q} | f-G_{\mathbf{u}_{}}| dv \right)^{\frac{1}{q}}} < +\infty$$
because $f$ has finite moments for $t>0$ by \Cref{lemmedesmoments}.
Since this is true for all $\mathbf{u}_{} \in \mathcal{S}$, we can choose 
\begin{align}
\label{ustar}
    \mathbf{u}_{*}(t)= r(D) \frac{\mathbf{u}_{f}(t)}{|\mathbf{u}_{f}(t)|}. 
\end{align}
Since $\lim_{t\rightarrow \infty}\mathrm{dist}(\mathbf{u}_{f},\mathcal{S})=0$, there exists $t_0>0$ such that for all $t\geq t_0$, $\mathbf{u}_{f}(t)\in \mathcal{N}$ where $\mathcal{N}$ is a neighborhood of $\mathcal{S}$ where  \eqref{P_L_M_B} applies. Hence, $\forall t \geq t_0$,
\begin{align*}
     |\mathbf{u}_{f}(t)-\mathbf{u}_{*}(t)|^2 =\mathrm{dist}(\mathbf{u}_{f}(t),\mathcal{S})^2\leq \frac{2}{\mu_1}(\mathcal{V}(\mathbf{u}_{f}(t))-\mathcal{V}_*) \leq \frac{2}{\mu_1}(\mathcal{F}[f(t,.)]- \mathcal{F}_*) 
\end{align*}
and 
\begin{align*}
    ||f(t,.)-G_{\mathbf{u}_{*}(t)} ||_{L^1(\mathbb{R}^d)}^{\frac{1}{p}} \leq \mathcal{H}[f(t,.)|G_{\mathbf{u}_{*}(t)}]^{\frac{1}{2p}} \leq \left( \frac{1}{D}(\mathcal{F}[f(t,.)]- \mathcal{F}_*)+ \frac{1}{2D} |\mathbf{u}_{f}(t)-\mathbf{u}_{*}(t)|^2\right)^{\frac{1}{2p}}
\end{align*}
by \eqref{eq100} and \eqref{eq22}.
Using \Cref{prop1}, we conclude that for all $\beta>0$, there exists a constant ~$N_{\beta}>0$ such that
\begin{align}
\label{deriv0}
    \left|  \frac{d\mathbf{u}_{f}}{dt}\right|\leq N_{\beta}t^{-\frac{\beta}{2p}}.
\end{align}
Integrating this inequality on $(t,+\infty)$, we obtain \Cref{convdeJf} with $\mathbf{u}_{\infty} \in \mathcal{S}$ because of \Cref{lemmejfvois}.
\end{proof}

\begin{proof}[Proof of \Cref{thm2}.]
Let $\mathbf{u}_{\infty}=\lim_{t\rightarrow \infty}\mathbf{u}_{f}(t)$ which is well defined by \Cref{convdeJf}. Using \eqref{eq22}, that is, 
\begin{align*}
    \mathcal{H}[f(t,.)|G_{\mathbf{u}_{\infty}}]=\frac{1}{D}\left(\mathcal{F}[f(t,.)]-\mathcal{F}_*+\frac{1}{2}|\mathbf{u}_{f}-\mathbf{u}_*|^2 \right),
\end{align*}
we conclude thanks to \Cref{convdeJf} and \Cref{prop1}.
\end{proof}

\end{section}

\begin{section}{Linearization of the evolution operator and proof of~\texorpdfstring{\Cref{thm3}}{thm3}.}
\label{sec:: section 5}
We know from the previous section that there exists an average speed limit $\mathbf{u}_{\infty} \in \mathcal{S}$ for $\mathbf{u}_{f}$ and that $f$ converges in relative entropy to $G_{\mathbf{u}_{\infty}}$. It is natural to consider the linearized evolution operator around the steady state $G_{\mathbf{u}_{\infty}}$. However, as it is mentioned in \cite{Li2021}, the linearized operator around $G_{\mathbf{u}_{\infty}}$ does not have a spectral gap if $\mathbf{u}_{f}$ converges to $\mathbf{u}_{\infty}$ tangentially to $\mathcal{S}$. We will avoid this problem by considering the evolution around $G_*:=G_{\mathbf{u}_*}$ with $\mathbf{u}_*$ defined by \eqref{ustar}. Note that $G_*$ depends on $t$.

\begin{subsection}{Perturbation around a stationary state.}

   First of all, let us define the quadratic forms associated to the free energy $\mathcal{F}$ and the Fisher information $\mathcal{I}$ as in \cite{Li2021}.
    For $\mathbf{u}_{} \in \mathcal{S}$ and
\begin{align*}
     g \in X_\mathbf{u}:=  \left\{ g \in L^2(\mathbb{R}^d,G_{\mathbf{u}_{}}dv) \text{ :} \int_{\mathbb{R}^d} g\,G_{\mathbf{u}_{\,}}dv=0 \right\},
\end{align*}
  we define:
    \begin{itemize}
        \item $Q_{1,\mathbf{u}_{}}[g]:= \lim_{\epsilon \rightarrow 0}{\frac{2}{\epsilon^2}\big( \mathcal{F}[G_{\mathbf{u}_{}}(1+ \epsilon g)]-\mathcal{F}_*\big)}=D \int_{\mathbb{R}^d} g^2 G_{\mathbf{u}_{}}dv -D^2 |\mathbf{v}_g|^2, $
        \item $Q_{2,\mathbf{u}_{}}[g]:= \lim_{\epsilon \rightarrow 0}{\frac{1}{\epsilon^2} \mathcal{I}[G_{\mathbf{u}_{}}(1+ \epsilon g)]}=D^2 \int_{\mathbb{R}^d} |\nabla g -\mathbf{v}_g|^2 G_{\mathbf{u}_{}}dv,$
    \end{itemize}
where \begin{align}
\label{vg}
    \mathbf{v}_g= \frac{1}{D} \int_{\mathbb{R}^d} vgG_{\mathbf{u}}dv.
\end{align} For $\mathbf{u} \in \mathcal{S}$ fixed, we define on $X_\mathbf{u} $ the bilinear form associated to the quadratic form $Q_{1,\mathbf{u}}$: 
 \begin{align}
 \label{definitionQ1}
    \langle g_1,g_2\rangle_{\mathbf{u}} = D \int_{\mathbb{R}^d}g_1g_2 G_{\mathbf{u}_{}}dv -D^2 v_{g_1}v_{g_2}.
    \end{align}
Since $G_{\mathbf{u}}$ is a minimum of $\mathcal{F}$, the bilinear form $\langle \cdot ,\cdot \rangle_{\mathbf{u}}$ is nonnegative. We denote by
\begin{align}
\label{normL2classique}
    \lVert g\rVert_{2,\mathbf{u}}:= \sqrt{\int_{\mathbb{R}^d}g^2\,G_\mathbf{u}\,dv}
\end{align} the classical weighted $L^2$ norm on $X_\mathbf{u}$. The following lemma is proved in \cite[Section~5.1]{Li2021} for $\mathbf{u}=0$. By adapting the computations, we get the following result for any  $\mathbf{u} \in \mathbb{R}^d.$

\begin{lemma}
   Let $\mathbf{u}\in \mathcal{S}$.  Then, $f$ is a solution of \eqref{PDE} if and only if $g=(f-G_\mathbf{u})G_\mathbf{u}^{-1}$ solves
   \begin{align}
\label{evoldeg}
    \partial_t g = \mathcal{L}_\mathbf{u}g + \mathcal{R}_\mathbf{u}g
\end{align}
with 
\begin{align}
\label{Lu}
    \mathcal{L}_\mathbf{u}g= D \Delta g +( \nabla \psi_{\alpha}+v-\mathbf{u})(\mathbf{v}_g-\nabla g)
\end{align}
and 
\begin{align}
\label{Ru}
    \mathcal{R}_\mathbf{u}g= -\mathbf{v}_g \cdot [D \nabla g-( \nabla \psi_{\alpha}+v-\mathbf{u})g].
\end{align}
Here, $\mathcal{L}_\mathbf{u}$ is a linear operator defined on $X_\mathbf{u}$ and $\mathcal{R}_\mathbf{u}$ is the nonlinear remaining term. We recall that ~$\mathbf{v}_g$ is defined by \eqref{vg} and $\psi_{\alpha}$ by \eqref{PDE}.
\end{lemma}

\begin{proposition}
\label{exuniciteL2}
    If $g_{\mathrm{in}}  \in X_\mathbf{u}$. Then there exists a unique distributional solution of \eqref{evoldeg} in $X_T:=  L^2(0,T;V) \cap  C([0,T),H)\cap H^1(0,T;V')$ with $H=X_\mathbf{u}$ and $V=H^1(\mathbb{R}^d,G_\mathbf{u}\,dx)$.
\end{proposition}

An important observation is that the linear operator $\mathcal{L}_\mathbf{u}$ is symmetric with respect to the scalar product~$\langle \cdot , \cdot \rangle_\mathbf{u}$. More precisely we have the following lemma which follows from several integration by parts and is detailed in \cite[Lemma~5.1]{Li2021}.
\begin{lemma}
\label{symmetrie}
For all $\mathbf{u} \in \mathcal{S}$ and $g,h \in H^1(\mathbb{R}^d,G_\mathbf{u}\,dv)$, we have,
 
  \[\langle\mathcal{L}_\mathbf{u}g,h\rangle_\mathbf{u}= -D^2 \int_{\mathbb{R}^d}  (\nabla g - \mathbf{v}_g)\cdot (\nabla h-v_h)\,G_{\mathbf{u}}\,dv
\text{ and } 
    \langle\mathcal{R}_\mathbf{u}g,h\rangle_\mathbf{u}= D^2\, \mathbf{v}_g \cdot\int_{\mathbb{R}^d}g\,(\nabla h- v_h)\,  G_{\mathbf{u}}\,dv.\]

\end{lemma}
\end{subsection}
\begin{subsection}{A coercivity result} 
\Cref{lemmeXingyu} is proved in \cite[Proposition~3.1]{Li2021} 
 and is a Poincaré type inequality.
\begin{lemma}
\label{lemmeXingyu}
    Consider $\mathbf{u} \in \mathcal{S}$. Assume that $g\in H^1(\mathbb{R}^d,G_\mathbf{u}\,dv)$ and $\int_{\mathbb{R}^d}g\, G_\mathbf{u} \, dv =0$ then
    \begin{align*}
        Q_{2,\mathbf{u}}[g]\geq \beta(D)^2\frac{(\mathbf{v}_g\cdot {\mathbf{u}})^2}{|\mathbf{v}_g|^2|\mathbf{u}|^2} Q_{1,\mathbf{u}}[g]
    \end{align*}
where $\beta(D)^2:= D\Lambda(D)\big(1-\kappa(D)\big)$ and 
\begin{align*}
    \kappa(D):= \frac{1}{r(D)^2} \int_{\mathbb{R}^d}|(v-\mathbf{u})\cdot \mathbf{u}|^2G_\mathbf{u}dv \in (0,1)
\end{align*}
for an arbitrary $\mathbf{u} \in \mathcal{S}$.
\end{lemma}
 \Cref{lemmenormequi} was proved in \cite[Lemma~3.1]{Li2021} for $D>D_*$ and $\mathbf{u}=0$. Here we adapt the proof for $D<D_*$.
\begin{lemma}
\label{lemmenormequi}
    Let $\mathbf{u}\in \mathcal{S}$ and $g \in L^2(\mathbb{R}^d,G_{\mathbf{u}}\,dv)$ such that $\int_{\mathbb{R}^d} g\,G_{\mathbf{u}}\;dv=0$. Then 
    \begin{align*}
         Q_{1,\mathbf{u}}[g]
    & \geq \eta(D)^2 \frac{(\mathbf{v}_g\cdot {\mathbf{u}})^2}{|\mathbf{v}_g|^2|\mathbf{u}|^2} \int_{\mathbb{R}^d}g^2G_{\mathbf{u}}dv,
    \end{align*}
where $\eta(D)^2=D\big(1-\kappa(D)\big)$ and $\kappa(D)$ as in \Cref{lemmeXingyu}.
\end{lemma}
\begin{proof}
   If $w=  \int_{\mathbb{R^d}}g\,v \,G_{\mathbf{u}}\,dv=D\,\mathbf{v}_g$ for some $\mathbf{u}\in \mathcal{S}$, then 
\begin{align*}
   |w|^2
    &= \int_{\mathbb{R}^d}{g(v-\mathbf{u})\cdot w\, G_{\mathbf{u}}\,dv}
    \leq \lVert g \rVert_{2,\mathbf{u}} \left( \int_{\mathbb{R}^d} |(v-\mathbf{u}).w|^2 G_{\mathbf{u}}\,dv\right)^\frac{1}{2}. \\
\end{align*}
Thanks to \cite[Corollary~2.2]{Li2021}, we have for all $w\in \mathbb{R}^d$,
\begin{align*}
    \frac{1}{D} \int_{\mathbb{R}^d}|v-\mathbf{u}\cdot w|^2G_\mathbf{u}dv= (\kappa(D)-1)(w\cdot e)^2+|w|^2 \text{ where } e=\mathbf{u}/r(D).
\end{align*}
So we obtain
\begin{align*}
   |w|^2 &\leq \lVert g \rVert_{2,\mathbf{u}}\left[ D\left( (\kappa(D)-1)|w.e|^2+ |w|^2\right)\right]^{\frac{1}{2}}\\
    & \leq  \lVert g \rVert_{2,\mathbf{u}}\left[ D\left(1- (1-\kappa(D))\frac{|w\cdot e|^2}{|w|^2}\right)\right]^{\frac{1}{2}} |w|,
\end{align*}

and
\begin{align*}
    \left|\int_{\mathbb{R^d}}{gvG_{\mathbf{u}}dv}\right|^2 \leq D\left(1- (1-\kappa(D))\frac{|w\cdot e|^2}{|w|^2}\right)\lVert g \rVert_{2,\mathbf{u}}^2,
\end{align*}

\begin{align*}
    Q_{1,\mathbf{u}}[g]&= D \int_{\mathbb{R}^d}g^2G_{\mathbf{u}}dv- \left| \int_{\mathbb{R}^d}gvG_{\mathbf{u}}dv \right|^2 \geq D(1- \kappa(D))\frac{|w\cdot e|^2}{|w|^2} \lVert g \rVert_{2,\mathbf{u}}^2.
\end{align*}
\end{proof}
Let us define 
\begin{align}
\label{Wk}
    W_k:= \int_{\mathbb{R}^d}|v|^kG_\mathbf{u}dv
\end{align}
 for $\mathbf{u} \in \mathcal{S}$ and observe that $W_k$ is independent of the chosen $\mathbf{u} \in \mathcal{S}$.

\begin{lemma}
\label{normequiH1}
    There exist positive constants $a(D)$ and $b(D)$ such that for all ~$g\in H^1(\mathbb{R}^d, G_{\mathbf{u}} \,dv)$ ~{satisfying $\int_{\mathbb{R}^d }g G_{\mathbf{u}} \, dv =0$,}
    \begin{align*}
        a(D)Q_{2,\mathbf{u}}[g]^{\frac{1}{2}}\leq \lVert\nabla g \rVert_{2,\mathbf{u}} \leq   b(D) \frac{|\mathbf{v}_g|^2|\mathbf{u}|^2}{(\mathbf{v}_g \cdot \mathbf{u})^2}Q_{2,\mathbf{u}}[g]^{\frac{1}{2}}.
    \end{align*}
\end{lemma}

\begin{proof}
  \begin{equation}
\label{ptiteq}
\begin{aligned}
Q_{2,\mathbf{u}}[g] &= D^2 \lVert \nabla g -\mathbf{v}_g\rVert_{2,\mathbf{u}}^2\\
&= D^2 \lVert \nabla g \rVert_{2,\mathbf{u}}^2 - 2D^2\, \mathbf{v}_g \cdot \int_{\mathbb{R}^d} \nabla g\, G_\mathbf{u}\,dv + D^2 |\mathbf{v}_g|^2 \\
&\geq D^2 \lVert \nabla g \rVert_{2,\mathbf{u}}^2- 2DW_2^{1/2} \lVert g \rVert_{2,\mathbf{u}} \lVert \nabla g \rVert_{2,\mathbf{u}}
\end{aligned}
\end{equation}

Hence, 
\begin{align*}
    P\left(\lVert \nabla g \rVert_{2,\mathbf{u}}\right) \leq 0,
\end{align*}
where $P $ is the following second degree polynomial:
$$P(X):= D^2X^2-2DW_2^{\frac{1}{2}}\lVert g \rVert_{2,\mathbf{u}}X-Q_{2,\mathbf{u}}[g].$$

Using that for $a>0,b>0,c>0, $
$$ a x^{2}-b x-c \leq 0 \Rightarrow x \leq \frac{b+\sqrt{b^{2}+4 a c}}{2 a} \leq \frac{b}{a}+\sqrt{\frac{c}{a}}
$$
with \Cref{lemmeXingyu} and \Cref{lemmenormequi}, we obtain
\begin{align*}
   \lVert \nabla g \rVert_{2,\mathbf{u}}\leq \frac{1}{D}  \left(2W_2^{\frac{1}{2}}\lVert g \rVert_{2,\mathbf{u}}+Q_{2,\mathbf{u}}[g]^{\frac{1}{2}} \right) \leq b(D) \frac{|\mathbf{v}_g|^2|\mathbf{u}|^2}{(\mathbf{v}_g \cdot \mathbf{u})^2}Q_{2,\mathbf{u}}[g]^{\frac{1}{2}},
\end{align*}
with \begin{align*}
    b(D)=\frac{1}{D}  \left(\frac{2 W_2^{\frac{1}{2}}}{\eta(D) \beta(D)}+1\right).
\end{align*}
Using Cauchy-Schwarz inequality and \eqref{Poincaréinequality} with \eqref{ptiteq}, we obtain 
\begin{align}
    a(D)^2 Q_{2,\mathbf{u}}[g]\leq  \lVert \nabla g \rVert_{2,\mathbf{u}}^2
\end{align} 
with 
\begin{equation*}
    a(D) = \left(D^2 + \frac{2D W_2^{1/2}}{\Lambda(D)^{1/2}} + \frac{W_2}{\Lambda(D)} \right)^{-1/2} \qedhere
\end{equation*}
\end{proof}
For $\epsilon>0$, let 
$$\mathcal{A}_{\epsilon}:= \left\{ g \in L^2(\mathbb{R}^d,G_\mathbf{u}dv) \text{ : } \ \int_{\mathbb{R}^d}g\,G_\mathbf{u}\,dv=0 \ and \  |\mathbf{v}_g \cdot \mathbf{u}| \geq \epsilon |\mathbf{v}_g| \right\}.$$ 
\Cref{lemmenormequi} states that the norms $Q_{1,\mathbf{u}}$ and $\lVert . \rVert_{2,\mathbf{u}}$  are equivalent on $\mathcal{A}_{\epsilon}$. \Cref{normequiH1} states that the norms associated to $Q_{2,\mathbf{u}}$ and the canonical norm on $H^1(\mathbb{R}^d,G_\mathbf{u}\,dv)$ are equivalent on $\mathcal{A}_{\epsilon}$.
Let $\mathbf{u}\in \mathcal{S}$ and $g $ be a solution of \eqref{evoldeg}.
Using  \Cref{symmetrie}, we obtain
\begin{align}
\label{derivéPS}
    \frac{1}{2}\frac{d}{dt}\langle g,g\rangle_{\mathbf{u}}&=\langle  \partial_t g,g\rangle_{\mathbf{u}} 
    = \langle \mathcal{L}g+ \mathcal{R}g,g\rangle_{\mathbf{u}} 
    = - Q_{2,\mathbf{u}_{}}[g] + D^2 \mathbf{v}_g  \int_{\mathbb{R}^d}g(\nabla g- \mathbf{v}_g)G_{\mathbf{u}}dv.
\end{align}
If we forget the nonlinear part of \eqref{evoldeg} and consider that $g$ is a solution of 
\begin{align*}
    \partial_tg= \mathcal{L}g,
\end{align*}
then
\begin{align*}
     \frac{1}{2}\frac{d}{dt} Q_{1,\mathbf{u}}[g]=- \,Q_{2,\mathbf{u}_{}}[g].
\end{align*}
Since we proved $\lim_{t\rightarrow \infty}\mathbf{u}_{f}(t)=\mathbf{u}_{\infty}$ in \Cref{convdeJf}, a natural choice of $\mathbf{u}$ is $\mathbf{u}_{\infty}$.
If $g$ belongs to some $\mathcal{A}_{\epsilon}$ with $\mathbf{u}=\mathbf{u}_{\infty}$ for all time $t\geq 0$, we are able to establish an exponential decay rate of  $Q_{1,\mathbf{u}_\infty}[g]$ thanks to \Cref{lemmenormequi} and Grönwall's lemma. Since $\mathbf{v}_g= \frac{1}{D}(\mathbf{u}_{f}-\mathbf{u}_{\infty})$, this condition is equivalent to ~$|(\mathbf{u}_{f}-\mathbf{u}_{\infty}) \cdot \mathbf{u}_{\infty} | \geq \epsilon  |\mathbf{u}_{f}-\mathbf{u}_{\infty}| $.  It is not satisfied  if $\mathbf{u}_{f}$ converges to $\mathbf{u}_{\infty}$ tangentially to $\mathcal{S}$.

\end{subsection}
\begin{subsection}{Perturbation around the projected state on \texorpdfstring{$\mathcal{S}$}{S}.}

To improve coercivity, we consider a perturbation relative to the projected velocity,
\begin{align}
\label{gstar}
    f=G_{*}(1+g),
\end{align}
 where $\mathbf{u}_*$ defined by \eqref{ustar}
  denotes the projection of $\mathbf{u}_{f}$ on $\mathcal{S}$ and $G_*(t):= G_{\mathbf{u}_*(t)}$. From now on, we take $\mathbf{u}=\mathbf{u}_*$. According to \eqref{vg}, $\mathbf{v}_g= \frac{1}{D}\int_{\mathbb{R}^d}g\,v\,G_*\,dv.$

   \begin{lemma}
 \label{lemmestar}
    Under the assumptions of \Cref{thm3}, if $f$ solves \eqref{PDE} and $g$ and $\mathbf{u}_*$ are defined by \eqref{gstar} and \eqref{ustar}, then, $g(t)$  is a function in $H^1(\mathbb{R}^d,G_{*}(t)dv)$ for all $t>0$ and
     \begin{align}
     \label{evoldegstar}
         \partial_tg=\mathcal{L}_{*}g+\mathcal{R}_{*}g-\frac{v\cdot \mathbf{u}_*'}{D}(1+g)
     \end{align}
      where $\mathcal{L}_*:=\mathcal{L}_{\mathbf{u}_*}$ and $\mathcal{R}_*:=\mathcal{R}_{\mathbf{u}_*}$ are defined by \eqref{Lu} and \eqref{Ru}. Moreover there holds  $\mathbf{u}_*'\cdot \mathbf{v}_g=0$ where $\mathbf{u}_*':= \frac{\mathrm{d} \mathbf{u}_*}{\mathrm{d}t}$
 \end{lemma}
 \begin{proof}
 Let $t_0>0$ and $\mathbf{u}_0:=\mathbf{u}_*(t_0)$. Let $f(t)=G_{\mathbf{u}_0}(1+g^0(t))$ for $t\in \mathbb{R}$. By assumption of \Cref{thm3}, $f_{\mathrm{in}}\in \cap_{\mathbf{u}\in \mathcal{S}} L^2(\mathbb{R}^d,G_{\mathbf{u}}^{-1}dv)\subset L^2(\mathbb{R}^d,G_{\mathbf{u}_0}^{-1}dv)$. Hence, $g^0_{\mathrm{ini}}:=g^0(0)\in L^2(\mathbb{R}^d,G_{\mathbf{u}_0}^{}dv) .$
 Thanks to \Cref{exuniciteL2} we get $g^0 \in X_T$ hence, $$g(t_0)=g^0(t_0)\in H^1(\mathbb{R}^d,G_{\mathbf{u}_0 }dv).$$ 
 On one hand, using \eqref{evoldeg}, 
 \begin{align*}
     \partial_tf=G_{\mathbf{u}_0}\partial_tg^0=G_{\mathbf{u}_0}(\mathcal{L}_{\mathbf{u}_0}g^0+\mathcal{R}_{\mathbf{u}_0}g^0). 
 \end{align*}
 On the other hand, differentiating $f=G_{*}(1+g)$,
     \begin{align*}
         \partial_tf=\mathbf{u}_*'\nabla_{\mathbf{u}_*}(G_{\mathbf{u}_*})(1+g)+G_{*}\partial_tg.
     \end{align*}
    Using these equalities at time $t_0$, we obtain \eqref{evoldegstar} since 
    \begin{align}
     \label{nalblau}
         \nabla_{\mathbf{u}}G_{\mathbf{u}}(v)=\frac{v-\mathbf{u}}{D}G_{\mathbf{u}}(v)
     \end{align}
    and $\mathbf{u}_*'\cdot \mathbf{u}_*=0$. Since $\mathbf{u}_{f}=\mathbf{u}_*+\int_{\mathbb{R}^d}g \, v\, G_{*}\,dv$ and $\mathbf{u}_* \in \mathcal{S}$, we obtain $\mathbf{u}_*'\cdot \int_{\mathbb{R}^d}g\,v \,G_{*}\,dv=0$.
 \end{proof}

  The advantage of considering the perturbation around $G_{\mathbf{u}_*}$ instead of $G_{\mathbf{u}_{\infty}}$ is that $\mathbf{v}_g$ and $\mathbf{u}_*$ are always colinear, hence \Cref{lemmeXingyu}, \Cref{lemmenormequi} and \Cref{normequiH1} are rewritten as 
 \begin{align}
 \label{30}
     Q_{2,\mathbf{u}_*}[g]\geq \beta(D)^2Q_{1,\mathbf{u}_*}[g],
 \end{align}
 \begin{align}
 \label{31}
     Q_{1,\mathbf{u}_*}[g]\geq \eta(D)^2\lVert  g\rVert_{2,\mathbf{u}_*}^2,
 \end{align}
and 
 \begin{align}
 \label{32}
        a(D)Q_{2,\mathbf{u}_*}[g]^{\frac{1}{2}}\leq \lVert \nabla g\rVert_{2,\mathbf{u}_*} \leq   b(D) Q_{2,\mathbf{u}_*}[g]^{\frac{1}{2}}.
    \end{align}

 \begin{lemma}
\label{deriveps}
 Under the assumptions of \Cref{thm3}, if $f$ solves \eqref{PDE} and $g$ and $\mathbf{u}_*$ are defined by \eqref{gstar} and \eqref{ustar}, then 
    \begin{align}
        \frac{1}{2}\frac{d}{dt}\langle g,g\rangle _{\mathbf{u}_*}= -Q_{2,\mathbf{u}_*}[g]+ R[g]
    \end{align}
where 
\begin{align}
\label{expressiondeR}
    R[g]&=D^2 \mathbf{v}_g  \int_{\mathbb{R}^d}g(\nabla g- \mathbf{v}_g)G_{\mathbf{u}_*}\, dv -\frac{\mathbf{u}_*'}{2}\cdot \int_{\mathbb{R}^d}g^2vG_{\mathbf{u}_*}dv.
\end{align}
\end{lemma}
\begin{proof}
We have:
\begin{equation*}
\begin{split}
    \frac{1}{2}\frac{d}{dt}\langle g, g \rangle_{\mathbf{u}_*}
    &= \langle g, \partial_t g \rangle_{\mathbf{u}_*}
    + \frac{1}{2} \nabla_\mathbf{u} \langle g, g \rangle_\mathbf{u} \cdot \mathbf{u}_*'.
\end{split}
\end{equation*}
Using \eqref{evoldegstar}, we obtain,
\begin{equation}
\label{bigequa}
\begin{split}
    \frac{1}{2} \frac{d}{dt} \langle g, g \rangle_{\mathbf{u}_*}
    &= -Q_{2,\mathbf{u}_*}[g] 
    + D^2\, \mathbf{v}_g \cdot \int_{\mathbb{R}^d} g(\nabla g - \mathbf{v}_g) G_*\,dv \\
    &\quad - \frac{1}{D} \langle v \cdot \mathbf{u}_*'(1+g), g \rangle_{\mathbf{u}_*}
    + \frac{1}{2} \nabla_\mathbf{u} \langle g, g \rangle_\mathbf{u} \cdot \mathbf{u}_*'
\end{split}
\end{equation}
Let us start with the computation of $\langle v\cdot \mathbf{u}_*'(1+g),g\rangle_{\mathbf{u}_*}$.
By definition \eqref{definitionQ1},  
\begin{align*}
    \langle v\cdot \mathbf{u}_*',g\rangle_{\mathbf{u}_*}&=D\, \mathbf{u}_*'\cdot \int_{\mathbb{R}^d} v \, g\,G_{*}\,dv- D\int_{\mathbb{R}^d}(v\cdot \mathbf{u}_*') v  \,G_{*}\,dv \cdot \mathbf{v}_g.
\end{align*}
Assume $\mathbf{v}_g\neq 0$ and $\mathbf{u}_*'\neq0$. Since $\mathbf{v}_g \cdot \mathbf{u}_*'=0$, we can consider an orthonormal basis $\{e_1,e_2,..,e_d\}$ with $e_1=\mathbf{u}_*'/|\mathbf{u}_*'|$ and $e_2=\mathbf{v}_g/|\mathbf{v}_g|$. In these coordinates, by writing~$v=\sum_{i=1}^d v_ie_i$, we have $v\cdot \mathbf{u}_*'=v_1 |\mathbf{u}_*'|$ and $v\cdot \mathbf{v}_g=v_2|\mathbf{v}_g|$. Since the function $$v_1\mapsto v_1 \,v_2 \,G_{*} \left(\sum_{i=1}^dv_ie_i\right)=  v_1 \,v_2 \exp \left\{-\frac{1}{D}\left[\frac{\alpha }{4} \left(\sum_{i=1}^d|v_i|^2 \right)^2-\frac{\alpha}{2}\left( \sum_{i=1}^d |v_i^2|\right)+ \sum_{i=1}^d |v_i-\mathbf{u}_{f}\cdot e_i|^2\right] \right\}$$ is odd for all $(v_2,v_3,...v_d) \in \mathbb{R}^{d-1}$ because $\mathbf{u}_{f}\cdot e_i=0$ for $i\geq 2$, we obtain 
\begin{align*}
    |\mathbf{v}_g||\mathbf{u}_*'|\int_{\mathbb{R}^d}(v\cdot \mathbf{u}_*') v  \,G_{*}\,dv \cdot \mathbf{v}_g=\int \cdots \int  \, dv_2 \cdots dv_d
\left(\int_{\mathbb{R}}v_2 v_1  \,G_{*} \left(\sum_{i=1}^dv_ie_i\right)dv_1\right)=0,
\end{align*}
thanks to the change of variable $(v_1,v_2,...v_d)\mapsto \sum_{i=1}^dv_ie_i$.
Since $\mathbf{u}_*'\cdot \mathbf{v}_g=0$ we conclude 
$$\langle v\cdot \mathbf{u}_*',g\rangle_{\mathbf{u}_*}=0.$$
Hence, 

\begin{align}
\label{pop}
    \langle v\cdot \mathbf{u}_*'(1+g),g\rangle_{\mathbf{u}_*}=\langle v\cdot \mathbf{u}_*'g,g\rangle_{\mathbf{u}_*}=D\,\mathbf{u}_*' \cdot \int_{\mathbb{R}^d}v\,g^2G_{*}\,dv- D\mathbf{v}_g \cdot\int_{\mathbb{R}^d}(v\cdot \mathbf{u}_*')\,v\,g\,G_{*}\, dv.
\end{align}
Let us compute $\nabla_\mathbf{u}\langle g,g \rangle_\mathbf{u} $.
Using \eqref{definitionQ1} and the expression \eqref{nalblau}, we obtain:
\begin{align}
\label{bigequa2}
    \nabla_\mathbf{u}\langle g,g \rangle_\mathbf{u}&=D\,\int_{\mathbb{R}^d}g^2\,\nabla_\mathbf{u} G_\mathbf{u}\,dv- \nabla_\mathbf{u}\left| \notag \int_{\mathbb{R}^d}g\,v\,G_\mathbf{u}\,dv\right|^2\\&=\int_{\mathbb{R}^d}g^2(v-\mathbf{u})G_\mathbf{u}\,dv+2  \int_{\mathbb{R}^d}(v \otimes v)\, \mathbf{v}_g \,g\, G_\mathbf{u}\,dv-2\,D |\mathbf{v}_g|^2\mathbf{u},
\end{align}
where $v\otimes v$ denotes the matrix $(v_iv_j)_{1\leq i,j\leq d}$.\\
Finally,
since $\mathbf{u}_*'\cdot \mathbf{u}=0$, replacing \eqref{pop} and \eqref{bigequa2} in \eqref{bigequa},
\begin{align*}
    \frac{1}{2}\frac{d}{dt}\langle g,g\rangle_{\mathbf{u}_*}&=-Q_{2,\mathbf{u}_*}[g] + D^2 \mathbf{v}_g \cdot  \int_{\mathbb{R}^d}g(\nabla g- \mathbf{v}_g)G_{{*}}\,dv\\& - \left[\,\mathbf{u}_*' \cdot \int_{\mathbb{R}^d}v\,g^2G_{*}\,dv - \mathbf{v}_g \cdot\int_{\mathbb{R}^d}(v\cdot \mathbf{u}_*')\,v\,g\,G_{*}\,dv\right]  \\&+\frac{1}{2}  \mathbf{u}_*'\cdot \int_{\mathbb{R}^d}g^2\,v\,G_{*}\,dv +  \mathbf{u}_*'\cdot\int_{\mathbb{R}^d}(v\otimes v) \,\mathbf{v}_g\, g\, G_{*}\, dv.
\end{align*}
Using that $\mathbf{v}_g\cdot (v\cdot \mathbf{u}_*')v=\mathbf{u}_*'\cdot (v\otimes v)\mathbf{v}_g$, we obtain the result.
\end{proof}
\end{subsection}
\begin{subsection}{Control of the nonlinearity.}

\begin{lemma}
\label{lemmemom}
   There exists a constant $C_D>0$ such that for all $\mathbf{u} \in \mathcal{S}$ and $g \in H^1(\mathbb{R}^d, G_\mathbf{u}dv)$, there holds 
   \begin{align*}
       \frac{\alpha^2}{8D^2} \int_{\mathbb{R}^d}g^2|v|^6G_\mathbf{u}dv \leq C_D \int_{\mathbb{R}^d}g^2G_\mathbf{u}dv+ \int_{\mathbb{R}^d}|\nabla g|^2G_\mathbf{u}dv.
   \end{align*}
\end{lemma}

\begin{proof} 
     Let $h=g\,G_\mathbf{u}^{\frac{1}{2}}$. We recall that $G_\mathbf{u}$ is given by \eqref{Gibbsstate}. Let us denote ~${\phi_\mathbf{u}(v)= \frac{1}{D}(\frac{1}{2}|v-\mathbf{u}|^2+ \frac{\alpha}{4}|v|^4-\frac{\alpha}{2}|v|^2)}.$ We have,

    \begin{align*}
        \int_{\mathbb{R}^d}|\nabla g|^2 G_\mathbf{u}dv 
        &=  \int_{\mathbb{R}^d}\left(|\nabla h |^2+ \nabla h \cdot \nabla \phi_{\mathbf{u}} h + \frac{1}{4} |\nabla \phi_\mathbf{u}|^2h^2\right)dv \\
         &= \int_{\mathbb{R}^d}|\nabla h |^2dv + \int_{\mathbb{R}^d} h^2\left( \frac{1}{4} |\nabla \phi_\mathbf{u}|^2- \frac{1}{2} \Delta  \phi_\mathbf{u}\right)dv.
    \end{align*}
We can see that
\begin{align*}
    |\nabla \phi_\mathbf{u}|^2= \frac{\alpha^2}{D^2}|v|^6+ P_\mathbf{u}(v) \text{ and } \Delta  \phi_\mathbf{u}(v)= Q_\mathbf{u}(v)
\end{align*}
with $P_\mathbf{u} $ and $Q_\mathbf{u} $ respectively two polynomials (with $d$ variables) of degree at most than $5$ and $2$. Hence we can write
\begin{align*}
     \int_{\mathbb{R}^d}|\nabla g|^2 G_\mathbf{u}dv \geq \frac{\alpha^2}{4D^2}\int_{\mathbb{R}^d}g^2|v|^6 G_\mathbf{u}dv+\int_{\mathbb{R}^d}g^2S_\mathbf{u}(v) G_\mathbf{u}dv
\end{align*}
where $S_\mathbf{u}(v)= \frac{1}{4}P_\mathbf{u}(v)-\frac{1}{2}Q_\mathbf{u}(v)$ is a polynomial of degree at most than $5$.
\\
Let us take $R=R_D$ big enough in order that, for all $\mathbf{u} \in \mathcal{S}$ and all $|v| \geq R$, there holds
\[|S_\mathbf{u}(v)|\leq \frac{\alpha^2}{8D^2}|v|^6.  \]
In the end, we obtain 
\begin{align*}
     \int_{\mathbb{R}^d}|\nabla g|^2 G_\mathbf{u}dv 
     &\geq \frac{\alpha^2}{8D^2}\int_{\mathbb{R}^d}g^2|v|^6 G_\mathbf{u}dv - \sup_{v\in B_R}S_\mathbf{u}(v)\int_{\mathbb{R}^d}g^2 G_\mathbf{u}dv.
\end{align*}
Here $B_R$ denotes the ball of radius $R$ centered in $0$.
 Hence we obtain the result with the constant \mbox{~$C_D= \sup_{\mathbf{u}\in \mathcal{S}} \sup_{v\in B_{R_D} }| S_\mathbf{u}(v)|$.}
\end{proof}

\begin{lemma}
\label{contoleJ_*}
 Under the assumptions of  \Cref{thm3}. Assume that for all $t\geq 0$, we have $\mathbf{u}_{f}(t) \in \mathcal{N}$ , where $\mathcal{N}$ is defined by \eqref{N}, then
    \begin{align*}
        |\mathbf{u}_*'(t)| \leq \frac{2 \alpha }{r(D)\eta(D)}W_6^\frac{1}{2} Q_{1,\mathbf{u}_*}[g(t,.)]^\frac{1}{2} \text{ where } W_k  \text{ is defined by \eqref{Wk}}
    \end{align*}
\end{lemma}
\begin{proof}
    We have that $\mathbf{u}_*(t)= \pi(\mathbf{u}_{f}(t))$ where $\pi(v):= r(D) \frac{v}{|v|}$ is the projection on $\mathcal{S}$. Using the expression \eqref{derivéuf}, 
\begin{align*}
    \mathbf{u}_*'(t)&= d\pi_{\mathbf{u}_{f}(t)}\mathbf{u}_{f}'(t)
    = d\pi_{\mathbf{u}_{f}(t)}\left(- \alpha \mathbf{u}_{f}(t)+ \alpha \int_{\mathbb{R}^d}|v|^2vfdv \right)
    = d\pi_{\mathbf{u}_{f}(t)}\left(\alpha \int_{\mathbb{R}^d}|v|^2vgG_{*}dv\right),
\end{align*}
where $d\pi_v(w)=\frac{w}{|v|}-\frac{v\cdot w}{|v|^3}v$ denotes the differential of $\pi$. Since $\lVert d\pi_{\mathbf{u}_{f}}\rVert\leq \frac{1}{|\mathbf{u}_{f}|} \leq \frac{2}{r(D)}$, we obtain
\begin{align*}
    |\mathbf{u}_*'(t)| \leq \frac{2 \alpha }{r(D)} \int_{\mathbb{R}^d}|v|^3gG_{*}dv \leq  \frac{2 \alpha }{r(D)} \left(\int_{\mathbb{R}^d}|v|^6G_{*}dv\right)^\frac{1}{2} \lVert g \rVert_{2,\mathbf{u}_*}
    &\leq  \frac{2 \alpha }{r(D)\eta(D)}W_6^\frac{1}{2} Q_{1,\mathbf{u}_*}[g]^\frac{1}{2},
\end{align*}
by \Cref{lemmenormequi}.
\end{proof}
\begin{lemma}
\label{controlereste1} Under the assumptions of \Cref{thm3}, there exists a positive constant $B_D$ such that,
    \begin{align}
    \label{eq45}
        R[g(t,.)] \leq B_D \max(|\mathbf{v}_g(t)|,|\mathbf{u}_*'(t)|)Q_{2,\mathbf{u}_*}[g(t,.)] \,\, \, \forall t\geq0 \text{  where  } R[g(t,.)] \text{ is defined by } \eqref{expressiondeR} 
    \end{align}
Moreover, if $\mathbf{u}_{f}(t) \in \mathcal{N}$ for all $t\geq 0$, 
\begin{align}
    \label{controlereste} 
        R[g(t,.)] \leq K_D Q_{2,\mathbf{u}_*}[g(t,.)]Q_{1,\mathbf{u}_*}^{\frac{1}{2}}[g(t,.)]  \,\, \forall t\geq 0,
    \end{align}
with $K_D= B_D\left(\frac{W_2^{\frac{1}{2}}}{D \eta(D) }+ \frac{2 \alpha W_6^{\frac{1}{2}}}{r(D)\eta(D)}\right).$
\end{lemma}
\begin{proof}
        We will control both of the two terms of 
        \begin{align*}
              R[g]&=D^2 \mathbf{v}_g  \cdot \int_{\mathbb{R}^d}g(\nabla g- \mathbf{v}_g)G_{*}\,dv -\frac{\mathbf{u}_*'}{2}\cdot \int_{\mathbb{R}^d}g^2\,v\,G_{*}\,dv.
        \end{align*}
Let us compute  $ \mathbf{v}_g \cdot \int_{\mathbb{R}^d}\,g\,(\nabla g- \mathbf{v}_g)\,G_{{*}}$. 
Using \eqref{30} and \eqref{31},
\begin{align*}
        \left |D^2 \mathbf{v}_g \cdot \int_{\mathbb{R}^d}g(\nabla g- \mathbf{v}_g)G_{{*}}\,dv \right|&\leq D^2|\mathbf{v}_g| \lVert g \rVert_{2,\mathbf{u}_*} \lVert \nabla g -\mathbf{v}_g  \rVert_{2,\mathbf{u}_*}\\
        &  \leq \frac{D^2}{\eta(D) \beta(D)} |\mathbf{v}_g|Q_{2,\mathbf{u}_{*}}[g].\\
    \end{align*}
Let us compute $\int_{\mathbb{R}^d}g^2vG_{\mathbf{u}_*}dv.$
\begin{align}
\label{eenormeq}
   \left| \int_{\mathbb{R}^d}g^2vG_{{*}}dv \right | &\leq  \left( \int_{\mathbb{R}^d}g^2 |v|^6 G_{{*}}\,dv\right)^\frac{1}{6} \left( \int_{\mathbb{R}^d}g^2 G_{{*}}\,dv\right)^\frac{5}{6} \notag \\
   &\leq \left( \frac{8D^2}{\alpha}\right)^{\frac{1}{6}} \left( C_D\int_{\mathbb{R}^d}g^2  G_{{*}}dv+ \int_{\mathbb{R}^d}|\nabla g|^2  G_{{*}}dv\right)^\frac{1}{6} \left( \int_{\mathbb{R}^d}g^2 G_{{*}} dv\right)^\frac{5}{6}  \notag \\
   & \leq \gamma(D)Q_{2,\mathbf{u}_*}[g]
\end{align}
where 
\begin{align*}
    \gamma(D)=\left( \frac{8D^2}{\alpha}\right)^{\frac{1}{6}} \left[\left(\frac{C_D}{\eta(D)^2 \beta(D)^2}+b(D)^2\right)^{\frac{1}{6}} \left(\frac{1}{\eta(D)^2 \beta(D)^2} \right)^{\frac{5}{6}}\right].
\end{align*}
Here, we used \eqref{30}, \eqref{31}, \eqref{32} and  \Cref{lemmemom}.
Hence, for the second term of $R[g]$, we have
\begin{align*}
    \left|\frac{\mathbf{u}_*'}{2}\cdot \int_{\mathbb{R}^d}g^2vG_{*}dv\right| \leq \frac{\gamma(D) }{2} |\mathbf{u}_*'|Q_{2,\mathbf{u}_*}[g].
\end{align*}
This completes the proof of \eqref{eq45} with constant 
$$B_D:= \frac{D^2}{\eta(D) \beta(D)}+ \frac{\gamma(D)}{2}.$$
\Cref{controlereste} follows from \eqref{eq45}, using  \Cref{contoleJ_*} and the inequality 
\begin{equation*}
    |\mathbf{v}_g| \leq \frac{1}{D\eta(D)} W_2^{1/2} Q_{1,\mathbf{u}_*}[g]^{1/2}. \qedhere
\end{equation*}
\end{proof}

\end{subsection}

\begin{subsection}{Proof of  \texorpdfstring{\Cref{thme1}}{thme1} and \texorpdfstring{\Cref{thm3}}{thm3}.}
\label{subsec: preuvethme2}
Let us start by proving the local exponential decay of the solution $g$ of \eqref{evoldegstar} when $g$ is small.

 \begin{proposition}
\label{thme1}
Under the assumptions of \Cref{thm3} there exist some constants ~$A_D>0$ and \mbox{~$K_D>0$} such that if \mbox{$Q_{1,\mathbf{u}_*(0)}[g_{\mathrm{in}}]<A_D$}, then
\begin{align*}
    {Q_{1,\mathbf{u}_*}[g]}\leq\frac{Q_{1,\mathbf{u}_*}[g_{\mathrm{in}}]}{\left(1-K_D Q_{1,\mathbf{u}_*}[g_{\mathrm{in}}]^{\frac{1}{2}}\right)^2} e^{-2\beta(D)^2t} \, \, \, \forall t \geq 0
\end{align*}
 where $\beta(D)$ is defined in \Cref{lemmeXingyu}.
\end{proposition}

\begin{proof}[Proof of \Cref{thme1}]
Let $\epsilon>0$ and $\delta $ as in  \Cref{P-L}.
  Let \begin{align*}
      A_D:= \min\left\{\delta \,\eta(D)W_2^{-\frac{1}{2}},K_D^{-2}\right\} .
  \end{align*} 
First of all, let us check that $g$ satisfies the assumption under which   \eqref{controlereste} applies, that is 
\begin{align*}
    |\mathbf{u}_{f}(t)-\mathbf{u}_*(t)| \leq \delta  \, \,\forall t \geq 0.
\end{align*}
We have
\begin{align*}
    |\mathbf{u}_{f_{\mathrm{in}}}-\mathbf{u}_*|&= \left | \int_{\mathbb{R}^d }v\,g_{\mathrm{in}}\,G_{*}\,dv\right|
    \leq \frac{W_2^{\frac{1}{2}}}{\eta(D)}\,Q_{1,\mathbf{u}_*}[g_{\mathrm{in}}]<\delta.
    \end{align*}
 \\
For all  $t \geq 0$ such that $|\mathbf{u}_{f}(t)-\mathbf{u}_*(t)|<\delta $,
using \Cref{deriveps} and \eqref{controlereste}, we have:
 \begin{align*}
     \frac{d}{dt} Q_{1,\mathbf{u}_*}[g(t,.)]&= -\,2 Q_{2,\mathbf{u}_*}[g(t,.)]+2 R[g(t,.)] \\
     &\leq -\,2 Q_{2,\mathbf{u}_*}[g]+2 K_D Q_{2,\mathbf{u}_*}[g]  Q_{1,\mathbf{u}_*}[g]^{\frac{1}{2}} 
     \leq -\,2 Q_{2,\mathbf{u}_*}[g]\left(1-K_D  Q_{1,\mathbf{u}_*}[g]^{\frac{1}{2}} \right).
 \end{align*}
 Since $  Q_{1,\mathbf{u}_*}[g_{\mathrm{in}}]<A_D \leq \frac{1}{K_D^2}$,
 we obtain that $Q_{1,\mathbf{u}_*}[g]$ is nonincreasing hence, for $t\geq0$,
 \begin{align*}
    |\mathbf{u}_{f_{}}(t)-\mathbf{u}_*(t)|&= \left | \int_{\mathbb{R}^d }vg_{}G_{\mathbf{u}}dv\right|
    \leq \frac{W_2^{\frac{1}{2}}}{\eta(D)}Q_{1,\mathbf{u}_*}[g_{}] \leq\frac{W_2^{\frac{1}{2}}}{\eta(D)}Q_{1,\mathbf{u}_*}[g_{\mathrm{in}}]<\delta .
    \end{align*}
We conclude that $\mathbf{u}_{f}(t) \in \mathcal{N}$ for all time $t\geq 0$. Hence,
 for all $t\geq 0 $ we have
\begin{align*}
    \frac{d}{dt} Q_{1,\mathbf{u}_*}[g(t,.)]\leq -\,2 \beta(D)^2\,Q_{1,\mathbf{u}_*}[g(t,.)]\left(1-K_D  Q_{1,\mathbf{u}_*}[g(t,.)]^{\frac{1}{2}}\right),
\end{align*}
where we used \Cref{lemmeXingyu}. This differential inequality can be integrated with respect to $t$, which shows that 
\begin{equation*}
    Q_{1,\mathbf{u}_*}[g] \leq \frac{Q_{1,\mathbf{u}_*}[g_{\mathrm{in}}]}{\left(1 - K_D\, Q_{1,\mathbf{u}_*}[g_{\mathrm{in}}]^{1/2} \right)^2} \, e^{-2\beta(D)^2 t} \,\,\, \forall t \geq 0 \qedhere
\end{equation*}
\end{proof}
   \begin{proof}[Proof of \Cref{thm3}.]
   We recall that $g(t) \in H^1(\mathbb{R}^d,G_{*}(t))$ for all $t > 0$ thanks to \Cref{lemmestar}.
     First of all, let us prove that: 
   \begin{align*}
       \lim_{t\rightarrow \infty} Q_{1,\mathbf{u}_{*}(t)}[g(t)]=0.
   \end{align*}
 We first have, by \Cref{lemmejfvois}
\begin{align}
\label{limvg}
    \lim_{t\rightarrow \infty}D|\,\mathbf{v}_g(t)|=\lim_{t \rightarrow \infty}\left|\mathbf{u}_{f}(t)-\mathbf{u}_*(t)\right|=\lim_{t\rightarrow \infty}\mathrm{dist}(\mathbf{u}_{f}, \mathcal{S})  =0.
\end{align}
For $t$ big enough, $|\mathbf{u}_{f}(t)| \geq r(D)/2$ and we know from \eqref{deriv0} that $\lim_{t\rightarrow}\mathbf{u}_{f}'(t)=0$. Hence, using the inequality $ |\mathbf{u}_*'(t)| \leq 2{|\mathbf{u}_{f}'|}/{|\mathbf{u}_{f}|}$, we obtain by \Cref{lemmejfvois}
 \begin{align}
 \label{limuprime}
     \lim_{t\rightarrow \infty }|\mathbf{u}_*'(t)| =0.
 \end{align}
 Thanks to \eqref{limvg}, \eqref{limuprime} and \eqref{eq45}, we conclude that
  for all $\epsilon \in  (0,1)$, there exists a time $t_0>0$ such that for all  $t\geq t_0$, we have $R[g(t,.)]\leq \epsilon \, Q_{2,\mathbf{u}_*}[g(t,.)] $, and 
\begin{align*}
     \frac{d}{dt} Q_{1,\mathbf{u}_*}[g(t,.)]&= -\,2 Q_{2,\mathbf{u}_*}[g(t,.)]+2 R[g(t,.)] \\& \leq  -\,2(1-\epsilon )  Q_{2,\mathbf{u}_*}[g(t,.)]
     \leq -\,2\beta(D)^2(1 -\epsilon )Q_{1,\mathbf{u}_*}[g(t,.)].
\end{align*}
As a consequence, $Q_{1,\mathbf{u}_*}[g] $ converges to $0$. Hence for $t_0$ big enough, we can apply \Cref{thme1} to~$Q_{1,\mathbf{u}_*(t_0)}[g(t_0+.)]$ and we obtain
\begin{align*}
    {Q_{1,\mathbf{u}_*}[g(t)]}\leq\frac{Q_{1,\mathbf{u}_*}[g(t_0)]}{\left(1-K_D \,Q_{1,\mathbf{u}_*}[g{(t_0)}]^{\frac{1}{2}}\right)^2} \,e^{-2\beta(D)^2(t-t_0)} \text{ for all } t\geq t_0.
\end{align*}
   Using the inequality
   \begin{align*}
        f\log\left(\frac{f}{G_\mathbf{u}}\right)-(f-G_\mathbf{u}) \leq \frac{(f-G_\mathbf{u})^2}{G_\mathbf{u}},
   \end{align*}
   we obtain that for all $t\geq 0$
\begin{align}
\label{eq25}
    \mathcal{H}[f(t,.)|G_{*}(t)] \leq \int_{\mathbb{R}^d}|f-G_{*}|^2G_{*}^{-1}dv=\lVert g \rVert_{2,\mathbf{u}_*}^2 \leq \frac{1}{\eta(D)^2} Q_{1,\mathbf{u}_*}[g(t,.)] .
\end{align}
Let $\epsilon>0$ and $\mathcal{N}$ a neighborhood of $\mathcal{S}$ where \eqref{P_L_M_B} applies.  Using \Cref{lemmejfvois}, there exists a time $t_0 >0$, such that $\mathbf{u}_{f}(t) \in\mathcal{N}$ for all $t \geq t_0$. Using \eqref{equtile}, \eqref{eq22}  and \eqref{eq25} we have for $t\geq t_0$:
\begin{align}
\label{eq2}
    |\mathbf{u}_{f}(t)-\mathbf{u}_*(t)|^2 \leq \frac{2}{\mu}(\mathcal{V}(\mathbf{u}_{f})-\mathcal{V}_*)  \leq \frac{2}{\mu}(\mathcal{F}[f]-\mathcal{F}_*) \leq \frac{2D}{\mu}\mathcal{H}[f|G_{\mathbf{u}_*}] \leq \frac{2D}{\mu\eta(D)^2} Q_{1,\mathbf{u}_*}[g(t,.)].
\end{align}
Using \Cref{contoleJ_*}, we obtain for $t\geq t_0$
 \begin{align}
    \label{eq1}
         |\mathbf{u}_*(t)-\mathbf{u}_{\infty}|^2&\leq \left|\int_{t}^{\infty}\mathbf{u}_*'(s)ds\right|^2 \leq \frac{ 2\alpha W_6^{\frac{1}{2}}}{\beta(D)^2r(D)  \eta(D)} \frac{Q_{1,\mathbf{u}_*}[g(t_0)]}{\left(1-K_D Q_{1,\mathbf{u}_*}[g(t_0)]^{\frac{1}{2}}\right)^2} e^{-2\beta(D)^2(t-t_0)}.
    \end{align}
 Using \eqref{eq1} and  \eqref{eq2}, we obtain for $ t\geq t_0$
\begin{align}
\label{eqfin}
    |\mathbf{u}_{f}(t)-\mathbf{u}_{\infty}|^2 & \leq 2\left(|\mathbf{u}_{f}-\mathbf{u}_*|^2+|\mathbf{u}_*-\mathbf{u}_{\infty}|^2\right) \notag\\
    & \leq 2\left(\frac{2D}{\mu \eta(D)^2}+\frac{ 2\alpha W_6^{\frac{1}{2}}}{\beta(D)^2r(D)  \eta(D)} \right) \frac{Q_{1,\mathbf{u}_*}[g(t_0)]}{\left(1-K_D Q_{1,\mathbf{u}_*}[g(t_0)]^{\frac{1}{2}}\right)^2} e^{-2\beta(D)^2(t-t_0)}.
\end{align}
After taking into account \eqref{eq22} with $\mathbf{u}=\mathbf{u}_{\infty}$, \eqref{eqfin} and \eqref{eq2}, we obtain for $t\geq t_0$, 
\begin{align*}
    \mathcal{H}[f(t,.)|G_{\mathbf{u}_{\infty}}] &\leq \frac{1}{D}\left(\mathcal{F}[f]-\mathcal{F}_* +\frac{1}{2}|\mathbf{u}_{f}-\mathbf{u}_{\infty}|^2 \right)\\
    &\leq \left(\frac{1}{ \eta(D)^2}+\left(\frac{2}{\mu \eta(D)^2}+\frac{ 2\alpha W_6^{\frac{1}{2}}}{D\beta(D)^2r(D)  \eta(D)} \right) \right) \frac{Q_{1,\mathbf{u}_*}[g(t_0)]}{\left(1-K_D Q_{1,\mathbf{u}_*}[g(t_0)]^{\frac{1}{2}}\right)^2} e^{-2\beta(D)^2(t-t_0)}.
\end{align*}
This concludes the proof of  \Cref{thm3}.

\end{proof}
\end{subsection}
\end{section}
\bigskip\noindent{\bf Acknowledgements.}
This work has been supported by the Project~\emph{Conviviality} (ANR-23-CE40-0003) of the French National Research Agency. The author would like to thank Jean Dolbeault and Amic Frouvelle for their guidance and helpful discussions throughout this research, as well as the anonymous referee for valuable comments and suggestions.\\[4pt]

{\scriptsize\copyright\,\the\year\ by the authors. This paper may be reproduced, in its entirety, for non-commercial purposes. }

\bibliographystyle{plain}         
\bibliography{bibliography}


\end{section}
\end{document}